\documentclass[12pt]{article}
\usepackage{amsmath,amssymb,amsthm,amsfonts,amscd}
\usepackage{hyperref}
\usepackage[numeric]{amsrefs}
\usepackage{rotating} 
\input colordvi
\usepackage{graphicx}

\newtheorem{thm}[equation]{Theorem}
\newtheorem{cor}[equation]{Corollary}
\newtheorem{lem}[equation]{Lemma}
\newtheorem{rem}[equation]{Remark}

\newcommand{\thmref}[1]{Theorem~\ref{#1}}
\newcommand{\corref}[1]{Corollary~\ref{#1}}
\newcommand{\lemref}[1]{Lemma~\ref{#1}}
\newcommand{\secref}[1]{Section~\ref{#1}}

\numberwithin{equation}{section}

\renewcommand\a{\alpha}

\newcommand\g{\gamma}
\renewcommand\d{\delta}
\newcommand\e{\varepsilon}
\renewcommand\l{\lambda}

\newcommand\G{\Gamma}
\newcommand\f{\frac}
\newcommand\smallf[2]{{\textstyle{\frac{#1}{#2}}}}

\newcommand{\Z}{{\mathbb{Z}}}
\newcommand{\R}{{\mathbb{R}}}

\newcommand{\C}{{\mathbb{C}}}

\newcommand{\Sch}{{\mathcal{S}}}

\renewcommand\Re{\text{Re~}}
\renewcommand\Im{\mbox{Im~}}

\renewcommand\i{^{-1}}
\renewcommand\({\left(}
\renewcommand\){\right)}
\newcommand{\ttwo}[4]{
\(\begin{smallmatrix}{#1} & {#2}
\\ {#3} & {#4} \end{smallmatrix}\)}
\newcommand{\tthree}[9]{\(\begin{smallmatrix}{#1} & {#2} & {#3}
\\ {#4} & {#5} & {#6} \\ {#7} & {#8} & {#9} \end{smallmatrix}\)}
\newcommand{\bx}{\hfill$\square$\vspace{.6cm}}

\newcommand\srel[2]{\begin{smallmatrix} {#1} \\ {#2} \end{smallmatrix}}

\newcommand{\gobble}[1]{}
  \newcommand{\rangeref}[2]{%
    \ref{#1}--\afterassignment\gobble\fam 0\ref{#2}%
  }

\def\sdmnote#1{ }

\makeatletter
\def\imod#1{\allowbreak\mkern5mu({\operator@font mod}\,#1)}
\makeatother

\begin{document}

\title{On the nonexistence of automorphic eigenfunctions of exponential growth on $SL(3,\Z)\backslash SL(3,\R)/SO(3,\R)$}

\date{November 2, 2018}

\author{Stephen D. Miller\thanks{Supported by NSF grants DMS-1500562 and DMS-1801417}\\
Rutgers University\\
\tt{miller@math.rutgers.edu}
\\
\\
Tien Trinh\\
Hanoi National University of Education\\
\tt{tientd@hnue.edu.vn}}

\maketitle

\begin{abstract}  It is well-known that there are automorphic eigenfunctions on $SL(2,\Z)\backslash SL(2,\R)/SO(2,\R)$ --- such as the classical $j$-function ---   that have exponential growth  and have exponentially growing Fourier coefficients (e.g., negative powers of $q=e^{2\pi i z}$, or an $I$-Bessel function).  We show that this phenomenon
 does not occur  on the quotient $SL(3,\Z)\backslash SL(3,\R)/SO(3,\R)$ and eigenvalues in general position (a removable technical assumption).

 More precisely, if such an automorphic eigenfunction has at most exponential growth, it cannot have non-decaying Whittaker functions in its Fourier expansion.  This confirms part of a conjecture of Miatello and Wallach, who assert all automorphic eigenfunctions on this quotient (among other rank $\ge  2$ examples) always have moderate growth.  We additionally confirm their conjecture under certain natural hypotheses, such as the absolute convergence of the eigenfunction's Fourier expansion.
\end{abstract}

\section{Introduction}\label{sec:intro}

Although the   $j$-function
\begin{equation}\label{jfunction}
    j(z) \ \ = \ \ e^{-2\pi i z}  \ + \ 744 \ + \ 196884\,e^{2\pi i z} \ + \ \cdots
\end{equation}
is one of the most prominent classical modular forms, it
 is excluded from the modern definition of {\em automorphic form} (see \cite{Borel,Bump-gray}) because it does not satisfy the moderate growth condition of being dominated by a polynomial in $\Im\!\!(z)$ for $\Im\!\!(z)$ large.  Put differently,
 an automorphic form on the upper half plane must be holomorphic at cusps, whereas the $j$-function is merely meromorphic.
 Indeed, the theory of modular forms on the complex upper half plane is replete with many such important examples, examples which have arithmetic significance despite not fitting into the standard representation-theoretic framework.
The exponential growth comes from the presence of nonzero Fourier coefficients for the Fourier modes $e^{2\pi i n z}$, $n<0$.

The situation for nonholomorphic Laplace eigenfunctions on the upper half plane is completely analogous.  For example, weak Maass forms (which do not have moderate growth but instead satisfy an exponential bound) for $SL(2,\Z)$ have Fourier expansions of the form
\begin{equation}\label{weakmaass}
  f(x+iy)= c_{+}y^{1/2+\nu} + c_{-}y^{1/2-\nu} + \sum_{n\in\Z_{\neq 0}}  e^{2\pi i n x} \sqrt{y} (a_n K_\nu(2\pi|n|y)+b_n I_\nu(2\pi|n|y))\,,
\end{equation}
where $\nu$, $c_\pm$, $a_n$, and $b_n$ are complex numbers\footnote{Strictly speaking, formula (\ref{weakmaass}) needs to be slightly adjusted when $\nu=0$, by multiplying the second term  by $\log(y)$.}, $K_\nu$ and $I_\nu$ are the Bessel functions defined in (\ref{IKdef}), and $b_n$ is nonzero for at most finitely many $n$.  Since the $K$-Bessel function decays exponentially and the $I$-Bessel function grows exponentially (see (\ref{IKnuasympt})), the condition that $f(x+iy)$ has moderate growth is equivalent to insisting $b_n=0$ for all $n$.

 The goal of this paper is to show that this prototypical $SL(2)$-phenomenon does not occur for $SL(3)$ (see Theorem~\ref{thm:main} below). Due to a technical limitation we fall slightly short of this, in that we must assume the Satake parameter  $\lambda=(\lambda_1,\lambda_2,\lambda_3)\in \C^3$ (the analog of $\nu$ -- see \secref{sec:Whittaker})
  satisfies
  \begin{equation}\label{generalposition}
  \lambda_i-\lambda_j \ \notin \ 2\Z \, , \ i\,\neq\,j\,.
\end{equation}
 Throughout this paper we make this standing assumption
so that we can quote   results about Whittaker functions.  That assumption is removable, but doing so here is impractical due to the space it would require to develop the theory of Whittaker functions in that context -- see the paragraph after (\ref{Mdef}) for further explanation.

  Our results are special cases of a  conjecture of Miatello and Wallach \cite{MW}, who posit that the   moderate growth condition is automatically satisfied for automorphic eigenfunctions on higher rank groups.\footnote{See \cite[p.~415]{MW} for a precise statement, which includes some reducibility conditions to rule out products of automorphic functions on rank one groups.}  The Miatello-Wallach conjecture is a generalization of the classical G\"otzky-K\"ocher principle \cites{Gotzky,Kocher}, which shows the moderate growth of {\em holomorphic} Hilbert and Siegel modular functions.  Holomorphy is used critically in those arguments, and these results can be understood in terms of Hartog's theorem on the impossibility of isolated singularities for holomorphic functions of several complex variables.

The  G\"otzky-K\"ocher principle can also be understood directly using Fourier expansions, an approach which Miatello and Wallach successfully used to prove their conjecture in the more complicated setting of nonholomorphic automorphic eigenfunctions on products of hyperbolic space.
Their argument essentially uses a factorization of the relevant automorphic coefficients, and deduces the inconsistency of two types of different behavior that must occur for eigenfunctions lacking moderate growth.  However, this factorization  cannot work in general, and certainly fails for  $SL(n,\R)$ when $n>2$.

Equally as important, all of the these  Fourier expansion arguments   heavily depend on having an {\em abelian} unipotent radical, the lack of which is a  serious obstacle already for $SL(3,\R)$ --- where tools as simple as the absolute convergence of the Piatetski-Shapiro/Shalika expansion   (\ref{introSL3Fourexp}) are unavailable.  In more detail,
automorphic forms on $SL(3,\Z)\backslash SL(3,\R)/SO(3,\R)$ have a Fourier expansion of the form
\begin{equation}\label{introSL3Fourexp}
\aligned
   F(g) \ \ &  = \ \ \sum_{k\,\in\,\Z}[P^{k,0,0}F](g) \ + \ \sum_{\ell\,=\,1}^\infty\,\sum_{\g\in\G_\infty^{(2)}\backslash \G^{(2)}}\,\sum_{k\,\in\,\Z}[P^{k,0,\ell}F]\(\ttwo{\gamma}{0}{0}{1} g\) \\
  &  = \ \  \sum_{\ell\,\in\,\Z}[P^{0,0,\ell}F](g) \ + \ \sum_{k\,=\,1}^\infty\,\sum_{\g\in\G_\infty^{(2)}\backslash \G^{(2)}}\,\sum_{\ell\,\in\,\Z}[P^{k,0,\ell}F]\(\ttwo{1}{0}{0}{\g} g\),
\endaligned
\end{equation}
where  $\G^{(2)}=SL(2,\Z)$, $\G_\infty^{(2)}$ is its subgroup of unit upper triangular matrices, and  the coefficients $P^{k,0,\ell}F$ are defined in (\ref{proj7}) (they are characterized as finite linear combinations of certain special functions in \secref{sec:Whittaker}).  These sums are only guaranteed to converge in the order stated.


 When $F$ is assumed to have {\em moderate growth}, meaning that   \begin{equation}\label{thmpunch}
    \left|   F\(\tthree{1}{x}{z}{0}{1}{y}{0}{0}{1}\tthree{a_1}000{a_2}000{a_3}  \)\right| \ \  \le \ \ C\,(\smallf{a_1}{a_3})^{N}  \, , \  \ \ a_1\, \ge\, \smallf{\sqrt{3}}{2}a_2 \,\ge \,\smallf34 a_3\,>\,0\,,
\end{equation}
for some positive constants $C$ and $N$ depending only on $F$,\footnote{The particular choice of the   constants $\f{\sqrt{3}}{2}$ and $\f34=(\f{\sqrt{3}}{2})^2$ in (\ref{thmpunch}) comes from the fact that region described by the inequalities contains a fundamental domain for $SL(3,\Z)$, but is  not essential to the statement.
} the coefficients in (\ref{introSL3Fourexp}) have a particularly special form, with $P^{1,0,1}F$ equal to a scalar multiple of the {\em decaying} Whittaker function $W_\l(g)$ defined in (\ref{Wdef}).  This is precisely analogous to the condition that $b_n$ vanish in (\ref{weakmaass}).  However, in the absence of such an assumption $P^{k,0,\ell}F$ instead merely belongs to a 6-dimensional subspace of Whittaker functions, of which only the scalar multiplies of a translate of $W_\l(g)$ decay -- again, analogously to (\ref{weakmaass}).

By comparison with the $SL(2)$ situation, one might expect that the Miatello-Wallach conjecture is equivalent to the absence of non-decaying Whittaker functions in the Fourier expansion  (\ref{thmpunch}). Indeed one direction is clear -- this absence is necessary for moderately growing forms -- but  sufficiency is difficult to prove
 when the maximal unipotent subgroup is nonabelian.  In particular, we cannot rule out the possibility that there is a  counterexample to the Miatello-Wallach conjecture having only decaying Whittaker functions in its Fourier expansion.

  Thus our results mainly address the absence of non-decaying Whittaker functions.  Our first result shows that the presence of even a single non-decaying Whittaker function implies that the terms in (\ref{introSL3Fourexp}) are not bounded:

\begin{thm}\label{thm:convergence1}  Assume (\ref{generalposition}).
 Let $F\in C^{\infty}(SL(3,\Z)\backslash SL(3,\R)/SO(3,\R))$ be an eigenfunction of the full ring of bi-invariant differential operators on $SL(3,\R)$.
Suppose that some $[P^{k,0,\ell}F](g)$ in   (\ref{introSL3Fourexp}) does not have moderate growth.  Then for some $g\in SL(3,\R)$ one of the two Fourier expansions in (\ref{introSL3Fourexp})  must   contain unboundedly large terms, and in particular is not absolutely convergent.
\end{thm}

To our knowledge, there are no examples in the theory of automorphic functions of Fourier expansions having unbounded terms, much less ones that do not converge absolutely.   In  Lemmas~\ref{lem:Fouriercoeffbound} (part 3)) and \ref{lem:modgrowthnoMs} we prove stronger results on sums of decaying Whittaker functions that allow us to conclude the Miatello-Wallach conjecture under the assumption that the Fourier expansion has bounded terms:

\begin{cor}\label{cor:convergence2}
 Assume (\ref{generalposition}).
The  Miatello-Wallach conjecture is true for eigenfunctions
  $F\in C^{\infty}(SL(3,\Z)\backslash SL(3,\R)/SO(3,\R))$  of the full ring of bi-invariant differential operators on $SL(3,\R)$  for which the Fourier coefficients
\begin{equation}\label{thmbddness}
    [P^{k,0,\ell}F]\(\ttwo\g{0}{0}1 g\)\, \text{and} \ \ [P^{k,0,\ell}F]\(\ttwo1{0}{0}\g g\), \ \  \text{for}
          \  (k,\ell)\,\neq\,(0,0),\,\g\,\in\,\G_\infty^{(2)}\backslash \G^{(2)},
\end{equation}
from (\ref{introSL3Fourexp})
are bounded for any fixed $g\in SL(3,\R)$.
  That is, the boundedness of (\ref{thmbddness}) implies that $F$ satisfies the moderate growth condition (\ref{thmpunch}) for some positive constants $C$ and $N$ depending only on $F$.
\end{cor}

Having shown the Miatello-Wallach conjecture under the assumption of bounded Fourier expansions, we now return to the situation of (\ref{weakmaass}) and
  impose an exponential bound on $F$.
  Note that the Miatello-Wallach conjecture is more general in that it allows for super-exponential growth that is excluded by the usual definition of weak modular or Maass form.  Indeed,   although all noteworthy automorphic eigenfunctions for $SL(2)$ (such as $j(z)$) are bounded by some exponential in $\Im\!\!(z)$, there do exist holomorphic modular functions (such as $e^{j(z)}$) which are not.

  The following result shows that an exponential bound is sufficient to rule out non-decaying Whittaker functions,
and hence unlike (\ref{weakmaass}) from  the classical $SL(2)$ theory, there are no  eigenfunctions on $SL(3,\Z)\backslash SL(3,\R)/SO(3,\R)$ that have both exponential growth and growing Whittaker functions.
(However, once again we cannot   rule out the possibility that an exponentially growing automorphic eigenfunction for $SL(3,\Z)\backslash SL(3,\R)/SO(3,\R)$ has only decaying Whittaker functions in its Fourier expansion.)

\begin{thm}\label{thm:main}
 Assume (\ref{generalposition}).
Let $F\in C^{\infty}(SL(3,\Z)\backslash SL(3,\R)/SO(3,\R))$ be an eigenfunction of the full ring of bi-invariant differential operators on $SL(3,\R)$, and assume that
\begin{equation}\label{thmcond}
 \left|   F\(\tthree{1}{x}{z}{0}{1}{y}{0}{0}{1}\tthree{a_1}000{a_2}000{a_3}  \)\right| \ \  \le \ \
 C\,\exp(K(\smallf{a_1}{a_2}+\smallf{a_2}{a_3})) \, , \ \ a_1\ge \smallf{\sqrt{3}}{2}a_2 \ge \smallf34 a_3>0\,,
\end{equation}
for some positive constants $C$ and $K$ depending only on $F$.  Then for all integers $k$ and $\ell$,  $[P^{k,0,\ell}F](g)$ has moderate growth in $g$; that is,  $F$'s Fourier expansion (\ref{introSL3Fourexp}) cannot contain non-decaying Whittaker functions.
\end{thm}

In Section\sdmnote{Sections ~\ref{sec:generalfourexp} and}\ref{sec:SL3fourexp} we give some   background on Fourier expansions, culminating in the Piatetski-Shapiro/Shalika formula (\ref{introSL3Fourexp}).  Sections~\ref{sec:Whittaker} and \ref{sec:estimates} are devoted to estimates on Whittaker functions, in particular     recent results of Templier \cite{Templier} derived from Givental's integral representation of Whittaker functions   \cite{Givental}.  The proofs of Theorems~\ref{thm:convergence1} and \ref{thm:main}   are given in Sections~\ref{sec:proofofconverg} and \ref{sec:proofofbound}, respectively.  Finally,
\secref{sec:Hecke} contains some  material about Hecke actions on automorphic functions not having moderate growth; in particular, it outlines a potential reduction aimed at removing  assumption (\ref{thmcond}).

\vspace{.23cm}

  \noindent {\bf Acknowledgements:} The authors would like to thank Nolan Wallach for his generous discussions and advice, out of which key ideas emerged.  The authors would also like to thank Daniel Bump, Bill Casselman, Dorian Goldfeld, Peter Sarnak, Wilfried Schmid, Eric Stade,  Nicolas Templier, Akshay Venkatesh, and Gregg Zuckerman for their guidance on various aspects of growth estimates.

\section{Fourier expansions on $SL(3,\Z)\backslash SL(3,\R)/SO(3,\R)$}\label{sec:SL3fourexp}

In this section we derive Fourier expansions automorphic functions on $\Gamma\backslash G$, where   $G=SL(3,\R)\supset \G=SL(3,\Z)$ (see \cite{Bump} for a general reference).  We use the standard notation $N\subset G$ for the subgroup of unit upper triangular matrices, $A$ for the subgroup of positive diagonal matrices,
and $K=SO(3,\R)$.  The subgroup $A$ is parameterized as
$A=\{a_{y_1,y_2}|y_1,y_2>0\}$, where
\begin{equation}\label{ay1y2}
    a_{y_1,y_2} \ \  = \ \ \tthree{y_1^{2/3}\,y_2^{1/3}}000{y_1^{-1/3}y_2^{1/3}}000{y_1^{-1/3}y_2^{-2/3}}.
\end{equation}
The Iwasawa decomposition asserts that each element of $g$ can be uniquely decomposed as $g=na_{y_1,y_2}k$ for some $n\in N$, $y_1,y_2>0$, and $k\in K$.  Note that (\ref{ay1y2}) has a well-defined meaning as an element of $G$ for any $y_1,y_2\neq 0$, e.g., $a_{1,-1}=\tthree{-1}000{-1}0001$.

Let $F\in C^\infty(\G\backslash G)$ and define the projections
\begin{equation}\label{proj1}
    [P^{m,n}F](g) \ \ := \ \ \int_{(\Z\backslash \R)^2} F\(\tthree{1}{0}{z}{0}{1}{y}{0}{0}{1}g\)e^{-2\pi i(mz+ny)}\,dy\,dz
\end{equation}
for $m,n\in\Z$,
so that
\begin{equation}\label{proj2}
    F(g) \ \ = \ \ \sum_{m,n\,\in\,\Z}  [P^{m,n}F](g)
\end{equation}
is an absolutely convergent Fourier series.
The change of  variables $z=au+bv$, $y=cu+dv$ in (\ref{proj1}) yields the identity
\begin{equation}\label{proj3}
    [P^{m,n}F]\(\tthree ab0cd0001g\) \ \ = \ \ [P^{ma+nc,mb+nd}F](g)\,, \ \ \ttwo abcd\,\in\,SL(2,\Z)\,.
\end{equation}
In particular,
\begin{equation}\label{proj4}
    [P^{m,n}F](g) \ \ = \ \ [P^{0,\ell}F]\(\tthree ab0cd0001g\),
\end{equation}
where $\ell=\gcd(m,n)$, $c=\f{m}{\gcd(m,n)}$,  $d=\f{n}{\gcd(m,n)}$, and $a,b\in\Z$ are chosen so that $ad-bc=1$.

The identity (\ref{proj4}) has a number of significant implications. For example,
the smoothness of $F$ and the Riemann-Lebesgue Lemma imply that the Fourier coefficient
\begin{equation}\label{projdecay}
    [P^{0,\ell}F]\!\(\tthree ab0cd0001g\) \  \longrightarrow \  0 \ \ \ \text{as} \ \ c^2+d^2\rightarrow\infty\,;
\end{equation}
in fact the decay is faster than any negative power of $c^2+d^2$.
It also follows that
\begin{equation}\label{proj5}
    [P^{0,\ell}F]\!\(\tthree 110010001 g\) \ \ = \ \  [P^{0,\ell}F]( g)\,,
\end{equation}
demonstrating that  (\ref{proj4}) depends only on the left-$\Gamma_\infty^{(2)}$ coset of $\ttwo abcd$ (this also reflects the fact that (\ref{proj4}) is independent of the choice of integers $a$ and $b$ satisfying $ad-bc=1$).  This periodicity implies the absolutely convergent Fourier expansion
\begin{equation}\label{proj6}
     [P^{0,\ell}F]( g) \ \ = \ \ \sum_{k\,\in\,\Z} [P^{k,0,\ell}F](g)\,,
\end{equation}
where
\begin{equation}\label{proj7}
    [P^{k,0,\ell}F](g) \ \ := \ \ \int_{(\Z\backslash \R)^3} F\(\tthree{1}{x}{z}{0}{1}{y}{0}{0}{1}g\)e^{-2\pi i(k x+\ell y)}\,dx\,dy\,dz\,.
\end{equation}
In fact,
\begin{equation}\label{proj7b}
    \int_{(\Z\backslash \R)}F\(\tthree10z010001 g\)dz \ \ = \ \ \sum_{k,\ell\,\in\,\Z}[P^{k,0,\ell}F](g)\,,
\end{equation}
hence the righthand side represents the Fourier expansion of the smooth function on the left-hand side,  thereby demonstrating the absolute  convergence of this double sum.

Combining (\ref{proj2}), (\ref{proj4}), and (\ref{proj6}) results in the Piatetski-Shapiro/Shalika  Fourier expansion \cite{psmult,shalika}
\begin{equation}\label{SL3Fourexp}
    \aligned
    F(g) \ \ & = \ \ [P^{0,0}F](g) \ + \ \sum_{\ell\,=\,1}^\infty\,\sum_{\g\in\G_\infty^{(2)}\backslash \G^{(2)}}\,\sum_{k\,\in\,\Z}[P^{k,0,\ell}F]\(\ttwo{\gamma}{0}{0}{1} g\) \\
    & = \ \ \sum_{k\,\in\,\Z}[P^{k,0,0}F](g) \ + \ \sum_{\ell\,=\,1}^\infty\,\sum_{\g\in\G_\infty^{(2)}\backslash \G^{(2)}}\,\sum_{k\,\in\,\Z}[P^{k,0,\ell}F]\(\ttwo{\gamma}{0}{0}{1} g\),
    \endaligned
\end{equation}
i.e., the first line in (\ref{introSL3Fourexp}).  For later reference, if  $\g=\ttwo{\g_{11}}{\g_{12}}{\g_{21}}{\g_{22}}$ and $g$ is written in Iwasawa form as $g=\tthree 1xz01y001a_{y_1,y_2}k$,
with $y_1,y_2>0$, and $k\in SO(3,\R)$, a short $SL(2)$ calculation shows that
\begin{equation}\label{remainingellsum1}
    \tthree{\g_{11}}{\g_{12}}0{\g_{21}}{\g_{22}}0001 g  \ \ \in \ \ N
    a_{\f{y_1}{\d(\g,x+iy_1)^{2}},y_2 \d(\g,x+iy_1)}K\,,
\end{equation}
where $\d\(\ttwo{\g_{11}}{\g_{12}}{\g_{21}}{\g_{22}},\tau\)=|\g_{21}\tau+\g_{22}|$.

In the above derivation we chose to initially integrate the variables $y$ and $z$ in (\ref{proj1}).  Had we instead performed a Fourier expansion over the subgroup $\{\tthree 1xz010001|x,z\in \R\}$ of  $SL(3,\R)$, we would have arrived at the second line in (\ref{introSL3Fourexp}),
\begin{equation}\label{SL3Fourexpcontra}
F(g) \ \ = \ \  \sum_{\ell\,\in\,\Z}[P^{0,0,\ell}F](g) \ + \ \sum_{k\,=\,1}^\infty\,\sum_{\g\in\G_\infty^{(2)}\backslash \G^{(2)}}\,\sum_{\ell\,\in\,\Z}[P^{k,0,\ell}F]\(\ttwo{1}{0}{0}{\g} g\),
\end{equation}
which could also have been obtained from (\ref{SL3Fourexp}) via the contragredient map
\begin{equation}\label{contragredient}
   g \ \ \mapsto \ \  \tthree0010{-1}0100 (g^t)\i\tthree0010{-1}0100\i.
\end{equation}

\begin{rem}\label{convergenceremark}{\rm
It is a simple consequence of convergence of Fourier series on $\Z\backslash \R$ and $(\Z\backslash \R)^2$ that the sum in  (\ref{SL3Fourexp}) converges in the order stated.  For the same reason, (\ref{SL3Fourexp}) remains convergent if the sums over $\ell$ and $\g$ are interchanged (recall these arose from labeling the Fourier modes for $m$ and $n$ in (\ref{proj4})).  However, orthogonality of the terms in (\ref{SL3Fourexp}) was lost after the introduction of the $\g$-translate in (\ref{proj3}); put differently, the $x$-integration for the Fourier modes indexed by $k$ in (\ref{proj7})-(\ref{SL3Fourexp}) is taken over different domains for different $\gamma$.  In particular, it is not clear that the sum (\ref{SL3Fourexp}) is absolutely convergent.  This distinction is important, since   \corref{cor:convergence2} establishes the Miatello-Wallach conjecture for $SL(3,\Z)\backslash SL(3,\R)/SO(3,\R)$  under the assumption of absolute convergence.  It should be stressed that there appear to be no known examples of automorphic Fourier expansions which are not absolutely convergent.}
\end{rem}

\section{Spherical Whittaker functions for $SL(3)$}\label{sec:Whittaker}

 We shall now make the further assumption that $F$ is spherical, i.e., fixed under $SO(3,\R)$:
\begin{equation}\label{spherical}
 F(gs) \ \ =  \ \ F(g) \,,  \ \ \ \ \ s\,\in\,SO(3,\R)\,.
\end{equation}
Thus $P^{k,0,\ell}F$   obeys the  transformation law
 \begin{multline}\label{proj8}
     [P^{k,0,\ell}F]\!\(\tthree{1}{x}{z}{0}{1}{y}{0}{0}{1}gs\) \ \  = \ \
     e^{2\pi i(k x+\ell y)} \,[P^{k,0,\ell}F](g)\,, \\ \text{for~~}  x,y,z\,\in\,\R\,\text{~~and~~}s\,\in\,SO(3,\R)\,.
 \end{multline}
 By the Iwasawa decomposition, such a function is uniquely determined by its restriction to the subgroup $A=\{a_{y_1,y_2}|y_1,y_2>0\}\subset G$ of positive diagonal matrices.
We will also henceforth assume that $F$ is an eigenfunction of the full ring of  bi-invariant differential operators.  The rest of this section is devoted to describing  the eigenfunction solutions to (\ref{proj8}), along with some of their properties (see \cite{Bump,BumpHuntley,stade,ishiistade,To,VinoTakh} for more details, with \cite{Bump} again serving as a general reference).

Let $\frak a^*_\C$ denote the complex-valued linear functionals on   $A$'s Lie algebra $\frak a=\{$traceless $3\times 3$ diagonal, real matrices$\}$; under this implicit identification of $\frak a$ with a subspace of $\R^3$,
the elements  $\lambda$ of  $\frak a^*_\C$  are
 concretely   realized as  triples of complex numbers $(\l_1,\l_2,\l_3)\in \C^3$ such that $\l_1+\l_2+\l_3=0$.
Any $\lambda\in\frak a^*_\C$
  naturally lifts  to a character of $A$, written using exponential notation as
\begin{equation}\label{lambdacharonA}
   \tthree{a_1}000{a_2}000{a_3}^{\l} \ \ = \ \ a_1^{\l_1}\,a_2^{\l_2}\,a_3^{\l_3}\,.
\end{equation}
Let $a(g)$ denote the Iwasawa $A$-component of $g=nak\in SL(3,\R)$, where $n\in N$, $a=a(g)\in A$, and $k\in SO(3,\R)$.  The functions $g\mapsto a(g)^{\lambda+\rho}$, where $\rho=(1,0,-1)\in {\frak a}_\C^*$,
 are eigenfunctions of the full ring of  bi-invariant differential operators.  Let $\Omega$ denote the Weyl group of $SL(3,\R)$ with respect to $A$, which we identify with  the symmetric group $S_3$; it acts on $\lambda=(\l_1,\l_2,\l_3)$ by  permutating the indices.  The eigenvalues of $a(g)^{w\lambda+\rho}$ under any  bi-invariant differential operator are independent of $w\in \Omega$.
 Moreover, given {\em any} eigenfunction $F$ of the full ring of  bi-invariant differential operators, there exists a unique Weyl orbit $\Omega\lambda\in \frak a^*_\C$  such that $F$ and $a(g)^{\l+\rho}$ share the same eigenvalues under any   bi-invariant differential operator.  In particular, the automorphic eigenfunction $F$ uniquely determines such a coset $\Omega\l\in \Omega\backslash\frak a^*_\C$, known as its Satake parameter.

%
%
Fourier expansions for eigenfunctions on $SL(2,\Z)\backslash SL(2,\R)/SO(2,\R)$ involve the
  $I$-Bessel and $K$-Bessel functions
  \begin{equation}\label{IKdef}
  \aligned
 I_\nu(x) \ \  & = \ \ \sum_{n\,=\,0}^\infty \,\f{(x/2)^{\nu+2n}}{n!\,\G(n+\nu+1)} \\
 \text{and} \ \ \   \ \ \ \ K_\nu(x) \ \ & = \ \ \f{\pi}{2}\f{I_{-\nu}(x)-I_{\nu}(x)}{\sin(\pi \nu)}\,.
 \endaligned
\end{equation}
The $I$-Bessel function grows exponentially for large $x$, whereas the $K$-Bessel function decays exponentially:
\begin{equation}\label{IKnuasympt}
\aligned
    I_\nu(u) \ \ & = \ \ \sqrt{\f{1}{2\pi u}}\,e^{u} \ + \ O(u^{-3/2}e^{u})\,, \ \ u\,\rightarrow\,\infty\\
    K_\nu(u) \ \ & = \ \ \sqrt{\f{\pi}{2u}}\,e^{-u} \ + \ O(u^{-3/2}e^{-u})\,, \ \ u\,\rightarrow\,\infty\,.
\endaligned
\end{equation}
In particular, $I$-Bessel
functions appear precisely in automorphic eigenfunctions which disobey the   moderate growth condition.
We now present definitions of some analogs for $SL(3,\Z)\backslash SL(3,\R)/SO(3,\R)$.
 Following \cite[Prop.~6]{ishiistade},  consider the non-decaying Whittaker function
\begin{multline}\label{Mdef}
    {\cal M}_\lambda(a_{y_1,y_2}) \ \ = \ \
    \f{\pi^3}{\sin(\smallf{\pi}{2}(\l_1-\l_2))\,\sin(\smallf{\pi}{2}(\l_2-\l_3))\,
    \sin(\smallf{\pi}{2}(\l_3-\l_1))}
    \\ \times \, |y_1 y_2|\,
    \sum_{m\,=\,0}^\infty \f{(\pi |y_1|)^{m-\l_3/2}(\pi |y_2|)^{m+\l_1/2}}{m!\,\G(m+\f{\l_1-\l_3}{2}+1)}\,
    I_{m+(\l_1-\l_2)/2}(2\pi |y_1|)I_{m+(\l_2-\l_3)/2}(2\pi |y_2|)\,.
\end{multline}
The sum over $m$     converges absolutely  to an entire function of $\lambda$, and
 plays a role for $SL(3,\R)$ directly  analogous to that of $I_\nu$ for $SL(2,\R)$.
We extend ${\cal M}_\l$ to a function of $G$ via the transformation law (\ref{proj8}) with $(k,\ell)=(1,1)$. 

 Let ${\mathcal M}_\lambda^\circ$ denote the second line in (\ref{Mdef}).
The functions ${\mathcal M}_{w\lambda}^\circ$, $w\in \Omega$, are linearly independent
when (\ref{generalposition}) holds; this can be seen from their leading small-$y_i$ asymptotics.  However, when (\ref{generalposition}) fails (such as when both $\lambda_1-\lambda_2$ and $\lambda_2-\lambda_3$ are even integers), the dimension of the span of these functions can drop to as low as 1, as can be directly verified directly from the definition and the fact that $I_\nu=I_{-\nu}$ for integral $\nu$.  (Note that there is a slight mistake in the standard references   for $GL(3)$ Whittaker functions \cite[p.24]{Bump} and \cite[p.27]{BumpHuntley}, which assert the span is 6-dimensional whenever the $\lambda_i$ are merely distinct.)  The literature also currently lacks a description of  the other    eigenfunction solutions to (\ref{proj8}) we are about to describe -- not just Whittaker functions -- in this degenerate case.  We have elected to make the (slightly) restrictive assumption  (\ref{generalposition}) as a result of the impracticality of developing  such a theory here, which would significantly lengthen this paper.

Similarly, we define degenerate Whittaker functions
\begin{equation}\label{Mdegendef1}
    {\cal M}^{\a_1}_{\text{degen},\l}(a_{y_1,y_2}) \ \ = \ \  |y_1|^{1-\l_3/2}|y_2|^{1-\l_3}I_{ (\l_1-\l_2)/2}(2\pi  |y_1|)\,,
\end{equation}
which extend to  functions on $G$ via the transformation law (\ref{proj8}) with $(k,\ell)=(1,0)$, and
\begin{equation}\label{Mdegendef2}
    {\cal M}^{\a_2}_{\text{degen},\l}(a_{y_1,y_2}) \ \ = \ \ |y_1|^{ 1+\l_1}|y_2|^{1+\l_1/2}I_{(\l_2-\l_3)/2}(2\pi |y_2|)\,,
\end{equation}
which extend to  functions on $G$ via the transformation law (\ref{proj8}) with $(k,\ell)=(0,1)$. (The superscripts   refer to nondegenerate roots for the character in (\ref{proj8}).)  The two functions (\ref{Mdegendef1}) and (\ref{Mdegendef2}) are related by the contragredient map (\ref{contragredient}).
The linear combination
\begin{equation}\label{Wdegendef}
      W_{\text{degen},\lambda}^{\a_1}(g) \ \ = \ \ \f{\pi}{2} \,\f{ {\cal M}^{\a_1}_{\text{degen},\l}(g) \ - \  {\cal M}^{\a_1}_{\text{degen},(12)\l}(g)}{\sin(\smallf{\pi}{2}(\l_2-\l_1))}\,,
\end{equation}
where $(12)$ denotes the transposition permutation in $\Omega\cong S_3$,
has moderate growth; in fact, it decays rapidly in the $y_1\rightarrow\infty$ limit, as can been seen from the exact formula
\begin{equation}\label{Mdegen1d}
   W_{\text{degen},\lambda}^{\a_1}(a_{y_1,y_2})
    \ \ = \ \
    |y_1|^{1-\l_3/2}\,|y_2|^{1-\l_3}K_{ (\l_1-\l_2)/2}(2\pi  |y_1|)
\end{equation}
(a consequence of the
 second formula in (\ref{IKdef})). Likewise, we have
 \begin{equation}\label{Wdegendefa2}
 \aligned
      W_{\text{degen},\lambda}^{\a_2}(g) \ \ & = \ \ \f{\pi}{2} \,\f{ {\cal M}^{\a_2}_{\text{degen},\l}(g) \ - \  {\cal M}^{\a_2}_{\text{degen},(23)\l}(g)}{\sin(\smallf{\pi}{2}(\l_3-\l_2))} \\
      \text{and} \ \ \ \ \ \ \ \
      W_{\text{degen},\lambda}^{\a_2}(a_{y_1,y_2}) \ \  & = \ \ |y_1|^{ 1+\l_1}|y_2|^{1+\l_1/2}\,K_{(\l_2-\l_3)/2}(2\pi |y_2|)\,,
 \endaligned
\end{equation}
consistently with (\ref{contragredient}).

Before listing the  exact form of the eigenfunction solutions to (\ref{proj8}), it is important to recall that there are more solutions listed here than appear in the  classical $L^2$ setting (where one {\em assumes} polynomial growth rather than attempting to deduce it as we are here).  It will be useful to note that
\begin{equation}\label{Pklsigns}
    [P^{k,0,\ell}]F(a_{y_1,y_2}) \ \ = \ \   [P^{k,0,-\ell}F](a_{y_1,y_2}) \ \ = \ \   [P^{-k,0,\ell}]F(a_{y_1,y_2})\,,
\end{equation}
as can be seen from (\ref{proj7}) using  the invariance of $F$ under the elements $a_{1,-1} =\tthree{-1}000{-1}0001$ and $a_{-1,1}=\tthree 1000{-1}000{-1}$.  In particular, it suffices to specify $P^{k,0,\ell}F$ for $k,\ell\ge 0$.    In the following $\lambda$ remains a Satake parameter for $F$.

\subsection*{Both $k=\ell=0$}

 $[P^{0,0,0}F](g)$ is a linear combination $\sum_{w\in\Omega}c(0,0,w)a(g)^{w\l+\rho}$ for some coefficients $c(0,0,w)\in \C$ (the powers $a(g)^{w\l+\rho}$ are linearly independent by assumption (\ref{generalposition})).
 \sdmnote{Comments about interpretation when $\lambda$ is fixed by some $w\in \Omega$.]]\footnote{This is analogous to the appearance of the function $y^{1/2}\log(y)$ out of the pair of eigenfunctions $y^{s},y^{1-s}$ at $s=1/2$ for $SL(2,\R)$.}}

\subsection*{Precisely one of $k$ or $\ell$ vanishes}

If $\ell=0$ but $k\neq 0$, the solutions are linear combinations
\begin{equation}\label{Pk000}
    [P^{k,0,0}F](g) \ \ = \ \ \sum_{w\,\in\,\Omega} c(k,0;w)\,{\cal M}^{\a_1}_{\text{degen},w\l}(a_{k,1}\,g)
\end{equation}
 for some coefficients $c(k,0;w)\in \C$.  Since we have assumed (\ref{generalposition}), the functions ${\cal M}^{\a_1}_{\text{degen},w\l}(a_{k,1}\,g)$ are linearly independent as $w$ varies over $\Omega$.
When $P^{k,0,0}F$ has moderate growth (i.e., satisfies the upper bound in  (\ref{thmpunch})) one has that $c(k,0;(12)w)=-c(k,0;w)$ for all $w\in \Omega$, and vice-versa (cf.~(\ref{Wdegendef})-(\ref{Mdegen1d})).  In particular, $[P^{k,0,0}F](g)$ has moderate growth   if and only if
\begin{equation}\label{Pk000degen}
\gathered
\, \!
[P^{k,0,0}F](g)  \ \ = \ \ d(k,0;1)\, W_{\text{degen},\lambda}^{\a_1}(a_{k,1}\,g) \ +  \ d(k,0;2)\, W_{\text{degen},(123)\lambda}^{\a_1}(a_{k,1}\,g) \\  \qquad\qquad\qquad\qquad \ \ \ + \ d(k,0;3)\, W_{\text{degen},(321)\lambda}^{\a_1}(a_{k,1}\,g)
\endgathered
\end{equation}
 for some coefficients $d(k,0;1)$, $c(k,0;2)$, $d(k,0;3)\in \C$.

 Likewise,
 if $k=0$ but $\ell\neq 0$, the solutions are linear combinations
\begin{equation}\label{Pl000}
    [P^{0,0,\ell}F](g) \ \ = \ \ \sum_{w\,\in\,\Omega} c(0,\ell;w)\,{\cal M}^{\a_2}_{\text{degen},w\l}(a_{1,\ell}\,g)
\end{equation}
 for some coefficients $c(0,\ell;w)\in \C$, and the moderate growth of $P^{0,\ell,0}F$ is equivalent to  $c(0,\ell;(23)w)=-c(0,\ell;w)$ for all $w\in \Omega$. Thus
  $[P^{0,0,\ell}F](g)$ has moderate growth if and only if
\begin{equation}\label{Pk000degena2}
 \gathered
   \!\, [P^{0,0,\ell}F](g) \ \ = \ \ d(0,\ell;1)\, W_{\text{degen},\lambda}^{\a_2}(a_{1,\ell}\,g) \ +  \ d(0,\ell;2)\, W_{\text{degen},(123)\lambda}^{\a_2}(a_{1,\ell}\,g) \\
     \qquad\qquad\qquad\qquad \ \ \  + \ d(0,\ell;3)\, W_{\text{degen},(321)\lambda}^{\a_2}(a_{1,\ell}\,g)
 \endgathered
\end{equation}
 for some coefficients $d(0,\ell;1)$, $d(0,\ell;2)$, $d(0,\ell;3)\in \C$.

Actually, (\ref{Pk000degen}) and (\ref{Pk000degena2}) are implied not just by moderate growth of the Fourier coefficients, but also by anything slower than the exponential growth of (\ref{Mdegendef1})-(\ref{Mdegendef2}).

\subsection*{Both $k,\ell\neq 0$}

 If both $k,\ell\neq 0$, then the space of eigenfunctions satisfying the transformation law (\ref{proj8}) is 6-dimensional.  We first describe this space in the special case of $k=\ell=1$.
For any $w\in \Omega$, ${\cal M}_{w\lambda}(g)$ is also an eigenfunction solution to (\ref{proj8}).  Since we have assumed (\ref{generalposition}),   the 6 functions $\{{\cal M}_{w\l}(g)|w\in\Omega\}$ span the 6-dimensional space of Whittaker functions for $SL(3,\R)/SO(3,\R)$.

The asymptotics for $y_1,y_2\ge 1$ of linear combinations of $\{{\cal M}_{w\l}(a_{y_1,y_2})|w\in\Omega\}$ were conjectured by Gregg Zuckerman using an insightful connection to the WKB approximation of mathematical physics.  Zuckerman's conjectures were formulated more generally for $SL(n,\R)$; in our case of $SL(3,\R)$ they involve the six triples of algebraic functions in two variables $(p_1^{(m)}(y_1,y_2)$, $p_2^{(m)}(y_1,y_2)$, $p_3^{(m)}(y_1,y_2))$,   $1\le m \le 6$, for which \begin{equation}\label{nilpotency}
    \tthree{p_1^{(m)}(y_1,y_2)}{y_1^2}{0}{-1}{p_2^{(m)}(y_1,y_2)}{y_2^2}{0}{-1}{p_3^{(m)}(y_1,y_2)} \ \  \text{is a nilpotent matrix}\,.
\end{equation}
All six triples can be written explicitly; we shall number them   so that
\begin{equation}\label{pmlabels}
\aligned
    p_3^{(1)}(y_1,y_2)\,-\,p_1^{(1)}(y_1,y_2)  \ \ & = \ \  \ \ \, (y_1^{2/3}+y_2^{2/3})^{3/2} \,, \\
    p_3^{(2)}(y_1,y_2)\,-\,p_1^{(2)}(y_1,y_2)  \ \ & = \ \  -\, (y_1^{2/3}+e^{-2\pi i/3}y_2^{2/3})^{3/2} \,, \\
    p_3^{(3)}(y_1,y_2)\,-\,p_1^{(3)}(y_1,y_2)  \ \ & = \ \  -\, (y_1^{2/3}+e^{2\pi i/3}y_2^{2/3})^{3/2} \,, \\
    p_3^{(4)}(y_1,y_2)\,-\,p_1^{(4)}(y_1,y_2)  \ \ & = \ \  \ \ \, (y_1^{2/3}+e^{-2\pi i/3}y_2^{2/3})^{3/2} \,,\\
    p_3^{(5)}(y_1,y_2)\,-\,p_1^{(5)}(y_1,y_2)  \ \ & = \ \  \ \ \, (y_1^{2/3}+e^{2\pi i/3}y_2^{2/3})^{3/2} \,, \\ \text{and} \ \ \
    p_3^{(6)}(y_1,y_2)\,-\,p_1^{(6)}(y_1,y_2)  \ \ & = \ \  -\,(y_1^{2/3}+y_2^{2/3})^{3/2}\,.\\
\endaligned
\end{equation}
\sdmnote{Give statement of Zuckerman's conjecture.  Tien typed something about it.}
Parts of Zuckerman's conjecture were proven by To \cite{To} in his Ph.D.~Thesis.  Using  Givental's integral representation \cite{Givental}, Templier   \cite{Templier} has recently improved To's result:

\begin{thm}\label{thm:To&Nicolas}
\cite{To,Templier}
For any $\lambda\in {\frak a}_\C^*$, there is a basis $\{\phi^{(m)}_\lambda(g)|1\le m \le 6\}$ of eigenfunction solutions to (\ref{proj8}) with $k=\ell=1$ such that
\begin{equation}\label{ToandNicolasgrowth}
\aligned
   \log(\phi^{(m)}_\lambda(a_{y_1,y_2})) \ \ & \sim \ \ 2\,\pi  \(p_3^{(m)}(y_1,y_2)-p_1^{(m)}(y_1,y_2)\)
\endaligned
\end{equation}
for $y_1,y_2\ge1 $ (as either or both of $y_1,y_2\rightarrow\infty$).
\end{thm}

Each $\phi^{(m)}_\lambda$ can be written as a linear combination of $\{{\cal M}_{w\l}(g)|w\in \Omega\}$ with coefficients that are meromorphic in $\lambda$.
We write
 \begin{equation}\label{Wdef}
  W_\l(g) \ \ = \ \ \sum_{w\,\in\,\Omega}{\cal M}_{w\lambda}(g) \ \ = \ \ \phi^{(6)}(g)
 \end{equation}
 for the unique decaying solution, which appears prominently in the classical setting of moderate growth.  Vinogradov-Takhtajan \cite{VinoTakh} proved the integral formula
\begin{multline}\label{VTformula}
    W_\l(a_{y_1,y_2}) \ \ = \ \  4 \,|y_1|^{1-\l_2/2}|y_2|^{1+\l_2/2}\ \times \\
  \int_0^\infty K_{(\l_1-\l_3)/2}(2\pi |y_1| \sqrt{1+x})K_{(\l_1-\l_3)/2}(2\pi |y_2| \sqrt{1+x^{-1}})\,x^{-3\l_2/4}\,\f{dx}{x}\,.
\end{multline}
  It follows from this integral and the inequalities
$|K_{\nu}(x)|\le K_{\Re\!\!(\nu)}(x)>0$   that
\begin{equation}\label{Wrealbd}
    |W_\l(a_{y_1,y_2})| \ \ \le \ \ W_{\Re\!\!(\l)}(a_{y_1,y_2})\,,
\end{equation}
where the right-hand side is in fact positive.

\thmref{thm:To&Nicolas} describes the eigenfunction solutions to (\ref{proj8}) with $k=\ell=1$.  For $k,\ell\neq 0$, each function   $\phi^{(m)}_\lambda(a_{k,\ell}g)$ satisfies the transformation law (\ref{proj8}).  In light of this observation and (\ref{Pklsigns}),
we may thus write  $P^{k,0,\ell}F$ for $k,\ell\neq 0$ as
\begin{equation}\label{PlokandMwl}
    [P^{k,0,\ell}F](g) \ \ = \ \ \sum_{m\,=\,1}^6 c(k,\ell,m)\,\phi_\lambda^{(m)}(a_{k,\ell}\,g)\,.
\end{equation}
Each of the first five lines in (\ref{pmlabels}) has unbounded real part on the domain $\{y_1,y_2\ge 1\}$. Thus if $P^{k,0,\ell}(a_{y_1,y_2})$, $k,\ell\neq 0$, is bounded for $y_1,y_2 \ge 1$ (or in fact even if it merely has moderate growth -- or even any growth slower than that of the growing Whittaker functions), it must be a scalar multiple of the decaying Whittaker function (\ref{Wdef}),
\begin{equation}\label{Pk0ellW}
    [P^{k,0,\ell}F](g) \ \ = \ \ c(k,\ell)\,W_\l(a_{k,\ell}\,g)\,,
\end{equation}
 where $c(k,\ell)=c(k,\ell,6)=c(\pm k,\pm \ell)$ (see (\ref{Pklsigns})).

We close this section with a result from the second named author's Ph.D. thesis \cite{Trinh} on the asymptotics of $[P^{k,0,\ell}F](g)$ as $g$ varies along the image of a simple coroot:

\begin{thm}[Trinh \cite{Trinh}]\label{thm:Tien}
Let $k,\ell\neq 0$ and
suppose that $[P^{k,0,\ell}F](g)$ in (\ref{PlokandMwl}) satisfies the estimates
\begin{equation}\label{Tienassumptions}
    [P^{k,0,\ell}F](a_{t^{-2},t})\ ,   \ [P^{k,0,\ell}F](a_{t,t^{-2}})\ \ = \ \ O(1)
\end{equation}
for an infinite sequence of values of $t$ tending to $\infty$.
Then $[P^{k,0,\ell}F](g)$ is a multiple of the decaying Whittaker function $W_\lambda(a_{k,\ell}g)$, i.e., (\ref{Pk0ellW}) holds.
\end{thm}
For completeness, we include a sketch of the proof, starting with the well-known facts that $I_\nu(x)>0$ and $\frac{\partial}{\partial \nu}I_\nu(x)<0$ for positive values of $\nu$ and $x$.
Consider
 the integral representation
\begin{equation}\label{DLMF}
    I_\mu(x)\,I_\nu(x) \ \ = \ \ \f{2}{\pi}\int_0^{\pi/2}I_{\mu+\nu}(2\,x\,\cos\theta)\,\cos((\mu-\nu)\theta)\,d\theta \ , \ \ \
    \Re{\mu+\nu}>-1
\end{equation}
(\cite[10.32.15]{DLMF}), and specialize    $\mu=\sigma+it$ and $\nu=\bar{\mu}=\sigma-it$, where $\sigma,t\ge 0$.  Then differentiation under the integral sign shows that $\f{\partial}{\partial \sigma}|I_{\sigma+it}(x)|^2<0$ and  $\f{\partial}{\partial t}|I_{\sigma+it}(x)|^2>0$ for $\sigma>0$, i.e., $|I_\nu(x)|$ decreases in $\Re\!\!(\nu)$ and increases in $\Im\!\!(\nu)$ for $\Re\!\!(\nu)>0$.  In terms of definition (\ref{Mdef}), the large $t$-asymptotics of
\begin{multline}\label{Mdefttt1}
    {\mathcal{M}}_{\l}(a_{k t^{-2},\ell t}) \ \ = \ \
    \f{\pi^3\,|k|^{m+1-\l_3/2}\,|\ell|^{m+1+\l_1/2}}{\sin(\smallf{\pi}{2}(\l_1-\l_2))\,\sin(\smallf{\pi}{2}(\l_2-\l_3))\,\sin(\smallf{\pi}{2}(\l_3-\l_1))}
    \\ \times \,
    \sum_{m\,=\,0}^\infty \f{\pi^{2m+\l_1+\l_2/2} t^{-m-1+\l_3+\l_1/2}
   }{m!\,\G(m+\f{\l_1-\l_3}{2}+1)}\,
    I_{m+(\l_1-\l_2)/2}(2\pi |k|  t^{-2})I_{m+(\l_2-\l_3)/2}(2\pi |\ell| t)
\end{multline}
are thus manifest from (\ref{IKdef})-(\ref{IKnuasympt}) as
\begin{equation}\label{Mdefttt2}
    {\mathcal{M}}_{\l}(a_{kt^{-2},\ell t}) \ \ \sim \ \ d(k,\ell,\lambda)\,
      t^{-3/2(\l_1+1)} \,  e^{2\pi|\ell| t}
    \,,
\end{equation}
for some nonzero constant $d(k,\ell,\lambda)$ expressible as powers of $\pi$, $|k|$, and $|\ell|$.
Write $[P^{k,0,\ell}F](g)$ as a linear combination of the ${\mathcal{M}_{w\l}}(a_{k,\ell}g)$, $w\in \Omega$.  Since we have assumed (\ref{generalposition}), the $\lambda_i$ are distinct and the only way to avoid exponential growth in sums of (\ref{Mdefttt2})  is if the appropriate   proportionality of the coefficients of  ${\mathcal{M}_{w\l}}$ and ${\mathcal{M}_{(12)w\l}}$ holds.  The identical analysis applied to   ${\mathcal{M}}_{\l}(a_{kt,\ell t^{-2}})$ produces a  constraints between the coefficients of ${\mathcal{M}_{w\l}}$ and  ${\mathcal{M}_{(23)w\l}}$; combined, the coefficients are proportional to those in (\ref{Wdef}), forcing (\ref{Pk0ellW}) to hold.

\section{Estimates on decaying Whittaker functions and Fourier coefficients}\label{sec:estimates}

In this section we give  a rather crude estimate on the decaying spherical Whittaker function (\ref{Wdef}), and apply it to properties of the Fourier expansion (\ref{introSL3Fourexp}).   Far finer estimates are possible (see, for example, \cite[Theorem 1]{BumpHuntley}), but \lemref{lem:estimatesonW} --- which has a relatively simple derivation that we include  for completeness --- suffices for our purposes.

\begin{lem}\label{lem:estimatesonW}
1) There exist positive constants $Y_0$, $c_0$, and  $N$ depending on $\lambda$ such that
\begin{equation}\label{lemmalowerbd}
|\,W_\lambda(a_{y_1,y_2})\,| \ \  \gg \ \  (y_1 y_2)^{-N}\,e^{-c_0(y_1+y_2)} \ \ \gg \ \ e^{-2c_0(y_1+y_2)}
\end{equation}
 whenever $y_1,y_2>Y_0$, where the implied constant depends only on $\lambda$.

2) There exists an integer $N$ (again depending   on $\lambda$) such that
\begin{equation}\label{lemmaupperbd}
   |\,W_\lambda(a_{y_1,y_2})\,| \ \  \ll  \ \ (y_1^N+y_1^{-N})(y_2^N+y_2^{-N}) \, e^{-2\pi(y_1+y_2)} \ \ \ll \ \
   (y_1y_2)^{-N}\,e^{-\pi(y_1+y_2)}
\end{equation}
for any $y_1,y_2>0$, where the implied constant depends only on  $\lambda$.
\end{lem}

\begin{proof}

To simplify the calculations in this proof, we rescale  $y_1$ and $y_2$ by $2\pi$.
Since the terms outside the integral in (\ref{VTformula}) can be absorbed into the constant and polynomial factors, it suffices to consider the integral
\begin{equation}\label{estimatepf1}
    \int_0^\infty K_{\nu}(y_1\sqrt{1+x})\, K_{\nu}(y_2\sqrt{1+x^{-1}})\,x^{-3\l_2/4-1}\,dx\,,
\end{equation}
with  $\nu=(\l_1-\l_3)/2$.  

We begin with first estimate in  statement 2); the second estimate of course follows from it.
According to (\ref{IKnuasympt}), $e^uK_\nu(u)$ is bounded for large $u$.  Formula (\ref{IKdef}) shows that $K_\nu(u)$ is bounded by $|u|^{-q}$ for some $q>0$  as $u\rightarrow0$.  Consequently, $e^{y_1\sqrt{1+x}}K_{\nu}(y_1\sqrt{1+x})  e^{y_2\sqrt{1+x\i}}K_{\nu}(y_2\sqrt{1+x\i})$ is bounded by some fixed power of $(1+y_1\i)(1+y_2\i)$ for all $x\ge 0$, and   it thus suffices to show that
\begin{equation}\label{estimatepf2}
    \int_0^\infty \exp(-y_1\sqrt{1+x}-y_2\sqrt{1+x^{-1}})\,x^{p}\,dx \ \ \ll \ \ (y_1 y_2)^N\,e^{-y_1-y_2}
\end{equation}
for some $N>0$ and $p\in\R$.

We now split the range of integration into three pieces: $0<x<1/2$, $1/2\le x \le 2$, and $2<x<\infty$.  In the middle range,  the integrand is $O(e^{-y_1-y_2})$, as is its integral over $[1/2,2]$.  In the third range, the integrand is bounded by $e^{-y_1\sqrt{1+x}-y_2}x^{p}$.    Changing variables $x=u^2+2u$, we bound its integral over $(2,\infty)$ by
\begin{multline}\label{estimatepf3}
     e^{-y_1-y_2}\int_{\sqrt{3}-1}^\infty(u(u+2))^{p}\,(2u+2)\,e^{-y_1 u}\,du \ \ \ll \\
      e^{-y_1-y_2}\int_{\sqrt{3}-1}^\infty u^{2p+1}\,e^{-y_1 u}\,du \ \ \le \ \
      e^{-y_1-y_2}\, y_1^{-2p'-2}\,\G(2p'+2)\,,
\end{multline}
where $p'=\max(0,p)$ and the implied constant that depends only on $\l$.  Finally, the integral over the first range $0<x<1/2$ has the same form as that over the third range $2<x<\infty$, though with a different value of $p$. This
establishes  2).

   Assertion 1) follows by a similar analysis (or from \thmref{thm:To&Nicolas}).  We give a proof that results in the non-optimal value of $c_0=2\pi(\sqrt{3}+\sqrt{\smallf{3}{2}})$
   using the asymptotic lower bound for the $K$-Bessel function provided by (\ref{IKnuasympt}).  Rewrite (\ref{estimatepf1}) as
   \begin{multline}\label{estimatepf5}
    \int_0^\infty K_{\nu}(y_1\sqrt{1+x})\, K_{\nu}(y_2\sqrt{1+x^{-1}})\,x^{p}\,dx \ \ = \\
     \f{\pi}{2\sqrt{y_1y_2}}
     \(I \ + \ O(I_1) \ + \ O(I_2) \ + \ O(I_3)\)\,,
   \end{multline}
   where
   \begin{equation}\label{estimatepf6}
   \aligned
  I \ \ & = \ \ \int_0^\infty  \f{ \exp(-y_1\sqrt{1+x}-y_2\sqrt{1+x^{-1}})}{(1+x)^{1/4}\,(1+x\i)^{1/4}}\,x^{p}\,dx\,, \\
  I_1 \ \ & = \ \  \f{1}{y_1}\int_0^\infty   \f{\exp(-y_1\sqrt{1+x}-y_2\sqrt{1+x^{-1}})}{(1+x)^{3/4}\,
  (1+x\i)^{1/4}}\,x^{p}\,dx\,, \\
    I_2 \ \ & = \ \  \f{1}{y_2}\int_0^\infty  \f{\exp(-y_1\sqrt{1+x}-y_2\sqrt{1+x^{-1}})}{(1+x)^{1/4}\,
    (1+x\i)^{3/4}}\,x^{p}\,dx\,,   \ \ \text{and}\\
        I_3 \ \ & = \ \  \f{1}{y_1 y_2} \int_0^\infty  \f{ \exp(-y_1\sqrt{1+x}-y_2\sqrt{1+x^{-1}})}{(1+x)^{3/4}\,
        (1+x\i)^{3/4}}\,x^{p}\,dx  \\
  \endaligned
\end{equation}
and $p=-\frac{3\l_2}{4}-1$.
Noting that  $x\ge0$, we now choose a sufficiently large value of $Y_0$ so that    each of the three integrals $I_1$, $I_2$, and $I_3$ is bounded by $\f{1}{6}I$ for $y_1,y_2\ge Y_0$.  In this situation (\ref{estimatepf1}) is hence at least $\f{\pi}{4\sqrt{y_1y_2}}I$ in absolute value.  There is no loss of generality in assuming that $y_2\ge y_1\ge Y_0$, owing to the inherent symmetry present in (\ref{lemmalowerbd}) and in (\ref{estimatepf5}). Then by restricting the range of integration of $I$ in (\ref{estimatepf6}) to an interval, we obtain the lower bound
\begin{multline}\label{estimatepf7}
     \left| \int_0^\infty K_{\nu}(y_1\sqrt{1+x})\, K_{\nu}(y_2\sqrt{1+x^{-1}})\,x^{p}\,dx \right|  \ \  \gg \\ \f{1}{\sqrt{y_1 y_2}}\, \int_{(y_2/y_1)^{2/3}}^{1+(y_2/y_1)^{2/3}} \frac{e^{-y_1\sqrt{x+1} -y_2\sqrt{1+x\i}
    }}{\sqrt{x+1}}\,x^{p+\frac{1}{4}}\,dx\,.
\end{multline}
The exponent $-y_1\sqrt{x+1}-y_2\sqrt{x\i+1}$ in this last integral has a global maximum at $x=(y_2/y_1)^{2/3}$,  so
\begin{equation}\label{estimatepf8}
(\ref{estimatepf7}) \ \ \gg  \ \
  (y_1y_2)^{-1/2}\, (\smallf{y_2}{y_1})^{(4p-1)/6} e^{-y_2\,\sigma(y_2/y_1)}
   \,,
\end{equation}
where $\sigma(r)=\frac{\sqrt{r^{2/3}+2}}{r}+\sqrt{\frac{r^{2/3}+2}{r^{2/3}+1}}$.  Since
\begin{equation}\label{sigmaprime}
   \sigma'(r) \ \ = \ \ -\,\frac{r^{5/3}\left(r^{2/3}+1\right)^{-3/2}+2\, r^{2/3}+6}{3\,r^2 \,\sqrt{r^{2/3}+2} }
\end{equation}
is negative for positive values of $r$, $\sigma(y_2/y_1)\le \sigma(1)=\f{c_0}{2\pi}=\sqrt{3}+\sqrt{\frac{3}{2}}$.  Thus we conclude from (\ref{estimatepf7})  and (\ref{estimatepf8}) that
\begin{multline}\label{estimatepf9}
    \int_0^\infty K_{\nu}(y_1\sqrt{1+x})\, K_{\nu}(y_2\sqrt{1+x^{-1}})\,x^{p}\,dx \\ \gg \ \  (y_1y_2)^{-1/2}\,(\smallf{y_2}{y_1})^{(4p-1)/6} \,e^{-\sigma(1)\,y_2} \ \ \ge \ \
     (y_1y_2)^{-1/2}\, (\smallf{y_2}{y_1})^{(4p-1)/6}\, e^{-\sigma(1)\,(y_1+y_2)}
    \,,
\end{multline}
under the assumption that $y_2\ge y_1\ge Y_0$.  Finally,  the exponent $N$ in  (\ref{lemmalowerbd}) can be adjusted to absorb the  remaining  powers of $y_1$ and $y_2$ (as it was for those remaining from (\ref{VTformula}) at the beginning of the proof).
\end{proof}

Recall the
Fourier coefficients $c(k,\ell)$  defined in (\ref{Pk0ellW}),  for those $k,\ell\neq 0$ having $c(k,\ell,1)=\cdots=c(k,\ell,5)=0$ in (\ref{PlokandMwl}).  For   $k,\ell\neq 0$ for which the respective Fourier coefficient has moderate growth, we also defined coefficients $d(k,0;j)$ and $d(0,\ell;j)$, $j=1,2,3$, in (\ref{Pk000degen}) and (\ref{Pk000degena2}).  The following result crucially uses the lower bound in  part 1) of \lemref{lem:estimatesonW} to give upper bounds on these Fourier coefficients, which in turn  will be essential for  showing the moderate growth of Fourier expansions in Lemma~\ref{lem:modgrowthnoMs}.

\begin{lem}\label{lem:Fouriercoeffbound}

Assume (\ref{generalposition}).

1) The coefficient $c(k,\ell)$ is subexponential in $k$ and $\ell$, i.e., $c(k,\ell)=O_{\e}( e^{\e |k|+\e |\ell|})$
 for any fixed $\e>0$.

 2)  The coefficients $d(k,0;j)$, $j=1,2,3$, are subexponential in $|k|$, i.e., $d(k,0;j)=O_{\e}(e^{\e |k|})$ for any fixed $\e>0$.  Likewise, the coefficients $d(0,\ell;j)$, $j=1,2,3$, are subexponential in $|\ell|$.

3) Assume the boundedness of (\ref{thmbddness}) for any fixed $g\in SL(3,\R)$.  Then for any $\e>0$, $c(k,\ell) =O_{\e}(e^{\e \max(|k|,|\ell|)^{1/3}\min(|k|,|\ell|)^{2/3}})$.

\sdmnote{
4) Assume the growth estimate (\ref{thmcond}).  Then for any $\e>0$, $c(k,\ell)=O_{\e}(e^{\e |k|^{1/2} |\ell|})$ and $c(k,\ell)=O_{\e}(e^{\e |k|  |\ell|^{1/2}})$ this is rough; we don't use it now so I'll work on it later.}

\end{lem}
\begin{proof}
First we prove coefficient bounds using the fact that $[P^{k,0,\ell}F](g)$ is uniformly bounded on compacta in $g$, uniformly for $k,\ell\neq 0$, which is immediate from
applying absolute values to the integral in (\ref{proj7}).
First take $g$ to have the form  $g=a_{y_1,y_2}$, where   $y_2$ is at least the constant $Y_0$ guaranteed by part 1) of \lemref{lem:estimatesonW}.  Taking $y_1$ sufficiently small (depending on $\e$) and comparing (\ref{Pk0ellW}) with (\ref{lemmalowerbd}) results in the estimates
 \begin{equation}\label{oneell}
    c(k,\ell) \ = \ O_{\e,\ell}(e^{\e |k|})  \ \   \text{for fixed $\ell$} \ \ \   \ \ \text{and} \  \  \  \   c(k,\ell) \ = \ O_{\e,k}(e^{\e |\ell|}) \   \ \text{for fixed $k$}\,,
 \end{equation}
 the second bound following from the same logic but reversing the roles of $y_1$ and $y_2$.
  Taking both $y_1$ and $y_2$ to be sufficiently small gives the bound $c(k,\ell)=O_{\e}( e^{\e |k|+\e |\ell|})$ for sufficiently large $|k|$ and $|\ell|$.  Part 1) follows by combining these estimates.

The proof of part 2) is slightly more complicated because (\ref{Pk000degen}) is the sum of three different terms.  Inserting formula (\ref{Mdegen1d}) into it and evaluating at $g=a_{y_1,y_2}$ gives the formula
\begin{equation}\label{Fourcoeffbound1}
\aligned~~
    [P^{k,0,0}F](a_{y_1,y_2}) \ \ = \ \ &
    d(k,0;1)\,|k y_1|^{1-\l_3/2}\,|y_2|^{1-\l_3}\,K_{(\l_1-\l_2)/2}(2\pi|ky_1|)
    \\ &  \ + \
    d(k,0;2)\,|k y_1|^{1-\l_2/2}\,|y_2|^{1-\l_2}\,K_{(\l_3-\l_1)/2}(2\pi|ky_1|) \\ &  \ \  \ + \
    d(k,0;3)\,|k y_1|^{1-\l_1/2}\,|y_2|^{1-\l_1}\,K_{(\l_2-\l_3)/2}(2\pi|ky_1|)\,,
\endaligned
\end{equation}
an expression which is again bounded in $k$ when $y_1$ and $y_2$ are constrained to any fixed  compact set.
Each of the three terms on the righthand side is a power of $|y_2|$ times a function of $y_1$ and $k$.  Since the powers $|y_2|^{1-\l_j}$, $j=1,2,3$, are linearly independent due to our assumption (\ref{generalposition}), these three terms are each individually bounded in $k$ for any fixed $y_1,y_2>0$.  Then
taking $y_1$ to be arbitrarily small as in the proof of part 1) and using the lower bound in the asymptotics (\ref{IKnuasympt}) results in the bound claimed in part 2).

Because of the contragredient symmetry it suffices to prove 3) when $|k|\ge |\ell|$, in which case its estimate reads $c(k,\ell)=O_{\e}(e^{\e |k|^{1/3}|\ell|^{2/3}})$. The assumed boundedness and (\ref{Pk0ellW}) give the estimate
\begin{equation}\label{Fourcoeffbound2}
    \left|\,c(k,\ell)\,W_\l\(a_{k,\ell}\ttwo{\g}{0}{0}1 a_{y_1,y_2}\)\,\right| \ \ = \ \ O_{y_1,y_2}(1)
\end{equation}
for any $k,\ell\neq 0$ and $\g\in \G_\infty^{(2)}\backslash \G^{(2)}$.  Using (\ref{remainingellsum1}) and the transformation law (\ref{proj8}), the left-hand side has absolute value $|c(k,\ell)W_\l(a_{\f{|k|y_1}{\d(\g,iy_1)^2},|\ell|y_2 \d(\g,iy_1)})|$.  Take $y_1<1$ and  choose $\g\in SL(2,\Z)$ so that $\f{1}{2}|\f{k}{\ell}|^{1/3} < \d(\g,iy_1) < 2|\f{k}{\ell}|^{1/3}$.\footnote{The existence of such a $\g$ is equivalent to that of a relatively prime pair of integers $(c,d)$ for which $c^2y_1^2+d^2$ lies in the interval $(\f{1}{4}|\f{k}{\ell}|^{2/3},4|\f{k}{\ell}|^{2/3})$;   since $|k|\ge|\ell|$ and $y_1<1$,  one can take $(c,d)=(0,1)$  when $|\f{k}{\ell}|<8$ and   $(c,d)=(1,\lceil  |\f k\ell|^{1/3}\rceil)$  when $|\f{k}{\ell}|\ge 8$.}  Then
\begin{equation}\label{Fourcoeffbound3}
\aligned
 4\i\,|k|^{1/3}\,|\ell|^{2/3}\,y_1  \ \ &  < \ \  \ \ \ \ \f{|k|\,y_1}{\d(\g,iy_1)^2} \ \ & <  \ \ 4\,|k|^{1/3}\,|\ell|^{2/3}\,y_1
\\
 \text{and}
 \ \ \ \  \
  2\i\,|k|^{1/3}\,|\ell|^{2/3}\,y_2 \ \ &  < \  \ \ \ \ |\ell|\,y_2\, \d(\g,iy_1) \ \  & < \ \ 2\,|k|^{1/3}\,|\ell|^{2/3}\,y_2\,.
 \endaligned
\end{equation}
As $\ell\neq 0$, $|k|^{1/3}|\ell|^{2/3}\ge |k|^{1/3}$; hence for sufficiently large values of $|k|$ (depending only on $y_1$, $y_2$, and $\l$),
 part 1) of \lemref{lem:estimatesonW} guarantees  that
\begin{equation}\label{Fourcoeffbound4}
    \left|\,W_\l(a_{\f{|k|y_1}{\d(\g,iy_1)^2},|\ell|y_2 \d(\g,iy_1)})\,\right| \ \ \gg \ \ e^{-c_1|k|^{1/3}|\ell|^{2/3}(y_1+y_2)}
\end{equation}
for some constant $c_1>0$ depending only on $\lambda$.  The claimed estimate now follows from (\ref{Fourcoeffbound2}) by choosing small enough $y_1$ and $y_2$ so that $c_1(y_1+y_2)<\e$.

\sdmnote{We don't use part 4) so I'll fix the statement and complete  the proof later}

\end{proof}

It is clear that if $F\in C^\infty(SL(3,\Z)\backslash SL(3,\R)/SO(3,\R))$ is an automorphic eigenfunction of moderate growth, then the unipotent Fourier coefficients 
$[P^{k,0,\ell}F](g)$
defined  in (\ref{proj7}) all inherit this moderate growth.  The following Lemma presents a converse, but under the somewhat unsatisfactory assumption d) below on the growth of $F$'s abelian Fourier coefficients $c(k,\ell)$:

\begin{lem}\label{lem:modgrowthnoMs}
Suppose that an automorphic eigenfunction $F$ on the quotient $SL(3,\Z)\backslash SL(3,\R)/SO(3,\R))$ satisfies
\begin{enumerate}
\item[a)] (\ref{Pk000degen})  for all $k\neq 0$,
\item[b)] (\ref{Pk000degena2})  for all $\ell\neq 0$, and
\item[c)]    (\ref{Pk0ellW})  for all $k,\ell\neq 0$.
\end{enumerate}
Suppose furthermore that
\begin{enumerate}
\item[d)] $F$'s abelian Fourier coefficients $c(k,\ell)$ satisfy a bound of the form $c(k,\ell)=O(e^{\f{1}{8} \max(|k|,|\ell|)^{1/3}\min(|k|,|\ell|)^{2/3}})$.
\end{enumerate}
Then $F$ satisfies the moderate growth condition (\ref{thmpunch}).
\end{lem}

\noindent {\bf Remark:}  As we mentioned in the introduction, the presence of assumption d) prevents us from proving the full Miatello-Wallach conjecture for $SL(3,\Z)\backslash SL(3,\R)/SO(3,\R)$; instead, we only rule out non-decaying Whittaker functions.   Part 3) of \lemref{lem:Fouriercoeffbound} shows that when the terms in
 the Fourier expansion (\ref{introSL3Fourexp}) are bounded for any fixed $g$, then assumption d) holds -- in fact, with $\f 18$ replaced by any arbitrary positive constant.   The actual value of $\f 18$ in part d) is not optimal and was chosen to simplify the proof.
 Assumptions such as d) on the growth of Fourier coefficients do not appear in $SL(2)$ theory, nor in the situations considered by \cite{Gotzky,Kocher,MW}, because the relevant unipotent radicals there are abelian.

\begin{proof}
Consider the first line of (\ref{introSL3Fourexp}), keeping in mind the characterization of $[P^{k,0,\ell}F](g)$ from \secref{sec:Whittaker} in terms of linear combinations of special functions.
The assertion  (\ref{thmpunch}) is unchanged if $g$ is replaced by some left $N(\Z)$-translate, so write $g$ in its Iwasawa factorization as $g=n(g)a(g)k(g)$, where  $n(g)=\tthree 1xz01y001$ with $|x|\le 1/2$,   $a(g)=a_{y_1,y_2}$ with  $y_1,y_2\ge \smallf{\sqrt{3}}2$, and $k(g)\in K=SO(3,\R)$.

The term $[P^{0,0,0}F](g)$ is in the linear span of  $\{a(g)^{w\lambda+\rho}|w\in \Omega\}$ and hence has moderate growth.  Since (\ref{Pk000degen}) is assumed to hold for all $k\neq 0$, the estimate in  part 2) of \lemref{lem:Fouriercoeffbound},  formula (\ref{Fourcoeffbound1}), and the asymptotic (\ref{IKnuasympt}) together imply the estimate
\begin{equation}\label{modgrowthnoMs1}
    |[P^{k,0,0}F](g)| \ \ = \ \  |[P^{k,0,0}F](a_{y_1,y_2})| \ \ \ll_\e \ \ e^{\e |k|}\,|k\,y_1\,y_2|^p\,
    e^{-2\pi|k|y_1}\,,
\end{equation}
for any $\e>0$, where $p>0$ depends only on $\l$.  By taking $\e=1/2$, it follows
$\sum_{k\neq 0}|[P^{k,0,0}F](g)|$ has moderate growth for $y_1,y_2\ge \f{\sqrt{3}}{2}$.

Next we consider the  second sum in (\ref{introSL3Fourexp}), starting with terms corresponding to $\ell>0$ and $k=0$,
\begin{equation}\label{remainingsum9}
     \sum_{\ell\,=\,1}^\infty\,\sum_{\g\in\G_\infty^{(2)}\backslash \G^{(2)}} [P^{0,0,\ell}F]\(\ttwo{\gamma}{0}{0}{1} g\).
\end{equation}
Using (\ref{remainingellsum1}), (\ref{Wdegendefa2}), (\ref{Pk000degena2}), and  part 2) of \lemref{lem:Fouriercoeffbound}, we see that it is sufficient to show the moderate growth of
\begin{multline}\label{remainingsum10}
    \sum_{\ell\,=\,1}^\infty\,\sum_{\g\in\G_\infty^{(2)}\backslash \G^{(2)}} e^{\e|\ell|}\,W^{\a_2}_{\text{degen},\l}\(a_{1,\ell}\ttwo{\gamma}{0}{0}{1} g\) \ \ = \\
      \sum_{\ell\,=\,1}^\infty\,\sum_{\g\in\G_\infty^{(2)}\backslash \G^{(2)}}
      \frac{e^{\e|\ell|}
      (\smallf{y_1}{\d(\g,x+iy_1)^2})^{1+\l_1}}{ (|\ell|y_2\d(\g,x+iy_1))^{-1-\l_1/2}} K_{(\l_2-\l_3)/2}(2\pi|\ell|y_2\d(\g,x+iy_1) )
\end{multline}
for some choice of $\e>0$.
Inserting the upper bound from (\ref{IKnuasympt}) and
noting that the values of  $|\ell \d(\g,x+iy_1)|$ which occur are precisely the norms of nonzero vectors in the lattice spanned by $1$ and $x+iy_1$, this moderate growth is then immediate.

 The remaining terms are those in  the $\ell$-sum in (\ref{introSL3Fourexp})  having $k\neq 0$; we will show that their contribution has moderate growth for $y_1,y_2\ge \f{\sqrt{3}}{2}$.  Using (\ref{remainingellsum1}) and part 2) of \lemref{lem:estimatesonW}   results in the bound
\begin{multline}\label{remainingellsum2}
   |[P^{k,0,\ell}F](a_{k,\ell}\ttwo{\gamma}{0}{0}{1} g)|
    \ \ = \ \
    |c(k,\ell)\,W_\l\!\(a_{k,\ell}\ttwo{\gamma}{0}{0}{1} g\)| \\
    = \ \  \left|c(k,\ell)\, W_\l\!\(a_{\f{|k|y_1}{\d(\g,x+iy_1)^2},|\ell| y_2 \d(\g,x+iy_1)} \)\right| \\   \ll_{\e,N} \ \
    |c(k,\ell)|\,|\smallf{\d(\g,x+iy_1)}{k\ell y_1 y_2}|^{N}\,\exp\(-\pi\(\f{|k|y_1}{\d(\g,x+iy_1)^2}+|\ell| y_2 \d(\g,x+iy_1)\)\) \\
    \ll_{\e,N} \ \    |c(k,\ell)| \,\exp\(-\smallf{1}{4}\f{|k|y_1}{\d(\g,x+iy_1)^2}-\smallf{1}{4}|\ell| y_2 \d(\g,x+iy_1)\) , \ \
    y_1,y_2\ge \smallf{\sqrt{3}}{2},
\end{multline}
where  $N>0$ is sufficiently large (depending only on $\l$); here we have used 
$|\smallf{\d(\g,x+iy_1)}{k\ell y_1 y_2}|\le \frac{8}{3^{3/2}}|\ell y_2 \d(\g,x+iy_1)|$ and $-\f{1}{4}>-\f{\pi}{2}$.  Recall that $k$ is summed over nonzero integers,    $\ell$ is summed over positive integers, and  $\g$ is summed over the cosets for $\G^{(2)}_\infty\backslash \G^{(2)}$.  The subexponential estimate in part 1) of \lemref{lem:Fouriercoeffbound} guarantees that the $k,\ell$ sum for $\g=\ttwo 1001$,
    $\sum_{k\neq 0,\ell>0} |c(k,\ell)|\exp(-\smallf{1}{4}|k|y_1-\smallf{1}{4}|\ell|y_2)$, is bounded for $y_1,y_2\ge \f{\sqrt{3}}{2}$.

At this point assumption d) is crucially needed.
To simplify notation, index the coset representatives $\g=\g_{c,d}$ by their bottom row $[c\,d]$, and write $\d_{c,d}(z)=\d(\g_{c,d},z)=|cz+d|$, where $z=x+iy_1$.
Taking absolute values in (\ref{introSL3Fourexp})  and applying assumption d) to  (\ref{remainingellsum2}), we see that the moderate growth condition (\ref{thmpunch}) is implied by that for the sum
\begin{equation}\label{remainingsum3}
   \Sch \ \ = \ \  \sum_{k,\ell\,>\,0}\sum_{\srel{\g_{c,d}\,\in\,\G^{(2)}_\infty\backslash \G^{(2)}}{\g_{c,d}\,\not\in\,\Gamma_\infty^{(2)}}}e^{\f{1}{8} \max(k,\ell)^{1/3}\min(k,\ell)^{2/3}}e^{-(ky_1\d_{c,d}(z)^{-2}+\ell y_2 \d_{c,d}(z))/4}\,,
\end{equation}
where we have used the fact $c(k,\ell)=c(-k,\ell)$ (see the discussion following (\ref{Pklsigns})). 
Break up the sum $\Sch$ as $\Sch_1+\Sch_2+\Sch_3$, where
$\Sch_1$ is the sum restricted to $0<\ell\le k\le 27 y_2^3\d_{c,d}(z)^3$,
$\Sch_2$ is the sum restricted to $0<\ell\le k> 27 y_2^3\d_{c,d}(z)^3$, and
$\Sch_3$ is the sum restricted to $\ell> k>0$.

Let us first consider
\begin{equation}\label{remainingsum4}
    \Sch_1 \ \ = \ \ \sum_{\srel{\g_{c,d}\,\in\,\G^{(2)}_\infty\backslash \G^{(2)}}{\g_{c,d}\,\not\in\,\Gamma_\infty^{(2)}}}\sum_{\srel{0\,<\,k\,\le\,27y_2^3 \d_{c,d}(z)^3}{0\,<\,\ell\,\le\, k}}
    e^{\smallf{1}{8}k^{1/3}\ell^{2/3}-\smallf{1}{4}ky_1\d_{c,d}(z)^{-2}-\smallf{1}{4}\ell y_2 \d_{c,d}(z)}.
\end{equation}
The argument of the exponential is maximized in $\ell>0$ at $\ell=\f{k}{(3 y_2\d_{c,d}(z))^3}\le 1$,  hence it decreases for $\ell\ge 1$ and the summand is bounded by $\exp(\f{1}{8}k^{1/3}-\f{1}{4}ky_1\d_{c,d}(z)^{-2}-\f{1}{4}y_2\d_{c,d}(z))$.  This last expression is maximized in $k>0$ at $k=\f{\d_{c,d}(z)^3}{(6y_1)^{3/2}}$, so
\begin{equation}\label{remainingsum5}
\aligned
     \Sch_1 \ \ & \le \ \ \sum_{\srel{\g_{c,d}\,\in\,\G^{(2)}_\infty\backslash \G^{(2)}}{\g_{c,d}\,\not\in\,\Gamma_\infty^{(2)}}} (3\,y_2\,\d_{c,d}(z))^6\,\exp\(-\smallf{1}{72}(18y_2-\sqrt{\smallf{6}{y_1}})\d_{c,d}(z)\) \\
     & \le \ \ (3\,y_2)^6 \sum_{(c,d)\,\in\,\Z^2_{\neq(0,0)}}|c(x+iy_1)+d|^6\, e^{-|c(x+iy_1)+d|/12}
\endaligned
\end{equation}
for  $y_1,y_2\ge \f{\sqrt{3}}{2}$; thus $\mathcal S_1$ has moderate growth in this range.

Next,
the exponent   in
\begin{equation}\label{remainingsum6}
     \Sch_2 \ \ = \ \ \sum_{k\,>\,0}
     \sum_{\srel{\g_{c,d}\,\in\,\G^{(2)}_\infty\backslash \G^{(2)}}{\srel{\g_{c,d}\,\not\in\,\Gamma_\infty^{(2)}}{\d_{c,d}(z)\,<\,\f{k^{1/3}}{3y_2}}}}
     \sum_{0\,<\,\ell\,\le\,k}
     \exp\(\smallf{1}{8}k^{1/3}\ell^{2/3}-\smallf{1}{4}ky_1\d_{c,d}(z)^{-2}-\smallf{1}{4}\ell y_2 \d_{c,d}(z)\)
\end{equation}
is again maximized for $\ell>0$ at $\ell=\f{k}{(3 y_2\d_{c,d}(z))^3}$, where takes the value $-\f{k}{216 y_2^2 \d_{c,d}(z)^2}(54y_1y_2^2-1) \le -\f{ky_1}{8\d_{c,d}(z)^2}$ (since $27 y_1y_2^2>1$ for $y_1,y_2\ge \f{\sqrt{3}}{2}$).  As  $-\f{ky_1}{8\d_{c,d}(z)^2}<-\f 98 k^{1/3}y_1y_2^2$ for $\d_{c,d}(z)<\f{k^{1/3}}{3y_2}$,
\begin{equation}\label{remainingsum7a}
    \Sch_2 \ \ \le \ \
    \sum_{k\,>\,0} k\,N\!(x+iy_1,\smallf{k^{1/3}}{3y_2})\,\exp(-\smallf 98 k^{1/3}y_1y_2^2)\,,
\end{equation}
where $N\!(z,T)$ counts the number of $(c,d)\in \Z^2_{\neq(0,0)}$ for which $\d_{c,d}(z)^2=|(cx+d)^2+c^2y_1^2|<T^2$.  Since the quantity $\f{(cx+d)^2+c^2y_1^2}{c^2+d^2}$ is bounded below for $|x|\le 1/2$ and  $y_1\ge \f{\sqrt{3}}{2}$, $N\!(z,T)=O(T^2)$ and we conclude
\begin{equation}\label{remainingsum7b}
    \Sch_2 \ \ \ll \ \
    \sum_{k\,\neq\,0}\f{k^{5/3}}{y_2^2}\exp(-\smallf{27}{64}\sqrt{3}\,k^{1/3}) \ \ < \ \ \infty
\end{equation}
for $y_1,y_2\ge \f{\sqrt{3}}{2}$.

Finally,
\begin{multline}\label{remainingsum8}
   \Sch_3 \ \ = \ \ \sum_{\srel{\g_{c,d}\,\in\,\G^{(2)}_\infty\backslash \G^{(2)}}{\g_{c,d}\,\not\in\,\Gamma_\infty^{(2)}}}\sum_{\ell\,>\,0}\sum_{0\,<\,k\,<\,\ell}
    \exp\(\smallf{1}{8}k^{2/3}\ell^{1/3}-\smallf{1}{4}ky_1\d_{c,d}(z)^{-2}-\smallf{1}{4}\ell y_2 \d_{c,d}(z)\) \\
    <  \ \ \sum_{\srel{\g_{c,d}\,\in\,\G^{(2)}_\infty\backslash \G^{(2)}}{\g_{c,d}\,\not\in\,\Gamma_\infty^{(2)}}}\sum_{\ell\,>\,0}\sum_{0\,<\,k\,<\,\ell}
    \exp\(\smallf{1}{8}k^{2/3}\ell^{1/3}-\smallf{1}{4} \ell y_2 \d_{c,d}(z)\) \\
    <  \ \ \sum_{\srel{\g_{c,d}\,\in\,\G^{(2)}_\infty\backslash \G^{(2)}}{\g_{c,d}\,\not\in\,\Gamma_\infty^{(2)}}}\sum_{\ell\,\neq\,0}\ell
    \exp\(\smallf{1}{8}\ell^{2/3}\ell^{1/3}-\smallf{1}{4} \ell y_2 \d_{c,d}(z)\).
\end{multline}
This exponential's argument is $\f{\ell}{8}((1-\f{3}{2}y_2\d_{c,d}(z))-\f{1}{2}y_2\d_{c,d}(z))<-\f{\ell}{16}y_2\d_{c,d}(z)$ for $y_1,y_2\ge \f{\sqrt{3}}{2}$ (where we have used  $\d_{c,d}(z)\ge |c|y_1\ge \f{\sqrt{3}}{2}$ for $c\neq 0$).  Recalling that $\ell\d_{c,d}(z)$ is the norm of the nonzero lattice vector $\ell c(x+iy_1)+\ell d$, we conclude the last expression in (\ref{remainingsum8}) is  bounded for $y_1,y_2\ge \f{\sqrt{3}}{2}$.

\end{proof}

\section{Proof of \thmref{thm:convergence1}}\label{sec:proofofconverg}

Suppose to the contrary that for all fixed $g\in SL(3,\R)$ there is some constant bound on the terms in (\ref{introSL3Fourexp}), in particular  that the terms in (\ref{thmbddness}) are bounded when $g$ is taken to be the identity matrix.  We will show that (\ref{Pk000degen}), (\ref{Pk000degena2}), and (\ref{Pk0ellW}) hold, i.e., all $[P^{k,0,\ell}F](g)$ have moderate growth for $(k,\ell)\neq (0,0)$ as needs to be shown.  (Recall from \secref{sec:Whittaker} that  $[P^{0,0,0}F](g)$ always has moderate growth.)

From (\ref{remainingellsum1}) and the transformation properties in (\ref{proj8}) we have
\begin{equation}\label{proofofconverg1}
  \left|[P^{k,0,\ell}F]\(\ttwo{\g}001 g\)  \right| \  \ = \ \ \left|
  [P^{k,0,\ell}F]\( a_{\d(\g,i)^{-2}, \d(\g,i)}\)
  \right|,
\end{equation}
with  $\d\(\ttwo{\g_{11}}{\g_{12}}{\g_{21}}{\g_{22}},i\)=|\g_{21}i+\g_{22}|=\sqrt{\g_{21}^2+\g_{22}^2}$ growing to infinity as $\gamma$ varies. Using this and the contragredient symmetry gives a constant bound (for any fixed $(k,\ell)\neq(0,0)$) on $[P^{k,0,\ell}F](a_{t^{-2},t})$ and $[P^{k,0,\ell}F](a_{t,t^{-2}})$  over an infinite sequence of values of $t\rightarrow\infty$, which
Theorem~\ref{thm:Tien} shows implies (\ref{Pk0ellW}) when both $k$ and $\ell$ are nonzero.

It remains only to show (\ref{Pk000degena2}), since (\ref{Pk000degen}) is equivalent under the contragredient symmetry (\ref{contragredient}) which interchanges the roles of $k$ and $\ell$.   Thus we consider $k=0$ but $\ell\neq 0$, so that
  (\ref{Pl000}) exhibits $[P^{0,0,\ell}F](a_{y_1,y_2})$ as a linear combination $\sum_{w\in\Omega}c(0,\ell;w)
{\cal M}^{\a_2}_{\text{degen},w\l}(a_{y_1,\ell y_2})$  of the degenerate Whittaker functions
${\cal M}^{\a_2}_{\text{degen},w\l}$ defined in (\ref{Mdegendef2}).  Now,  (\ref{IKnuasympt}) and (\ref{Mdegendef2}) show  
\begin{equation}\label{5.2}
 {\cal M}^{\a_2}_{\text{degen},\l}(a_{t^{-2},\ell t}) \ \ \sim \ \ (2\pi)\i |\ell|^{(1+\l_1)/2} t^{-3(1+\l_1)/2}e^{2\pi |\ell| t}
\end{equation}
 for $t$ large.  Since we have assumed (\ref{generalposition}), the $\lambda_i$ are distinct; in order to cancel the exponential growth, the boundedness of (\ref{proofofconverg1}) forces   $c(0,\ell;(23)w)=-c(0,\ell;w)$ for all $w\in \Omega$, a condition   equivalent to (\ref{Pk000degena2}).
    \bx

\section{Proof of \thmref{thm:main}}\label{sec:proofofbound}

As in \secref{sec:proofofconverg}, we must again show (\ref{Pk000degen}), (\ref{Pk000degena2}), and (\ref{Pk0ellW}) for $(k,\ell)\neq (0,0)$.  In fact we need only prove the latter two assertions, i.e., the cases with $\ell \neq 0$, since the contragredient symmetry (\ref{contragredient}) interchanges the roles of $k$ and $\ell$ without affecting the assumptions and conclusions of \thmref{thm:main}.  The bound (\ref{thmcond}) on $F$ is  inherited by the  Fourier coefficients $P^{m,n}F$ and $P^{m,0,n}F$ through the unipotent integrations (\ref{proj1}) and (\ref{proj7}):
\begin{multline}\label{pfofbd1}
 \left| [P^{m,n}  F]\tthree{a_1}000{a_2}000{a_3}  \right|\,, \ \left| [P^{m,0,n}  F]\tthree{a_1}000{a_2}000{a_3}  \right| \\  \le \ \
 C\,\exp(K\smallf{a_1}{a_2}+K\smallf{a_2}{a_3}) \, , \ \ a_1 \ge \smallf{\sqrt{3}}{2}a_2\ge \smallf 34 a_3 \ \text{and} \ x\,\in\,\R\,,
\end{multline}
where $C$ and $K$ depend only on $F$.
In particular, this estimate is uniform in $m$ and $n$.\footnote{Of course the pointwise estimate   $[P^{m,n}F](g)\rightarrow 0$ as $m^2+n^2\rightarrow\infty$ holds by the Riemann-Lebesgue Lemma.  We shall not require this fact, but remark that the tension between this decay and the growth of Whittaker functions appears to be a fundamental reason behind the truth of the Miatello-Wallach conjecture.}
It follows by comparison of the   inequality on $P^{m,0,n}F$ with the growth rates in  (\ref{pmlabels}) and \thmref{thm:To&Nicolas}, that   (\ref{Pk0ellW}) holds for all but finitely many pairs $(k,\ell)$.  Likewise,   (\ref{IKnuasympt}), (\ref{Mdegendef2}), and (\ref{Pl000})  show that (\ref{Pk000degena2}) holds for all but finitely many $\ell$.

%
%

We now come to the key estimate of the argument.
If $\ell$ and $\g=\ttwo abcd\in SL(2,\Z)$ are related to $m$ and $n$ as in (\ref{proj4}), then
\begin{equation}\label{pfofbd2}
     [P^{0,\ell}F]\(\tthree ab0cd0001 \tthree{t}000{t}000{t^{-2}}\)
     \ \ =
     \ \
     [P^{m,n}F]\tthree{t}000{t}000{t^{-2}}
      \ \ \ll \ \  e^{K t^3}\,,  \ t> 1\,,
\end{equation}
with an implied  constant which is independent of $\ell$, $c$, and $d$.  The two matrices in the argument of $P^{0,\ell}F$ commute with each other, and so invoking (\ref{proj6}), (\ref{remainingellsum1}), and the $SO(3,\R)$ invariance of $P^{0,\ell}F$ results in the estimate
\begin{multline}\label{pfofbd34}
     [P^{0,\ell}F]\(\tthree{(c^2+d^2)^{-1/2}}{\theta_\g (c^2+d^2)^{1/2}}00{(c^2+d^2)^{1/2}}0001\tthree{t}000{t}000{t^{-2}}\) \ \ = \\
       \sum_{k\,\in\,\Z} [P^{k,0,\ell}F]\(\tthree{(c^2+d^2)^{-1/2}}{\theta_\g (c^2+d^2)^{1/2}}00{(c^2+d^2)^{1/2}}0001\tthree{t}000{t}000{t^{-2}}\) \ \ \ll \ \ e^{K t^3}\,,  \ t> 1
\end{multline}
where $\theta_\g=\f{ac+bd}{c^2+d^2}$ (coming from $\g$'s Iwasawa decomposition -- cf.~(\ref{remainingellsum1})).\footnote{This is readily deduced from $\ttwo{1}{\theta_\g}01 \ttwo{(c^2+d^2)\i}00{c^2+d^2}\ttwo 10{\theta_g}{1}= \g\g^t = \ttwo{a^2+b^2}{ac+bd}{ac+bd}{c^2+d^2}$.}  Roughly speaking, the $c^2+d^2$ factors  serve the amplify the growth rate of the left-hand side and cause it to violate (\ref{pfofbd34}), in a manner which we will make precise.

Fix  $\ell\neq 0$ and $(c,d)\neq (0,0)$.  Let $\mathcal S_\ell=\{k\neq 0|$(\ref{Pk0ellW}) does not hold for $(k,\ell)\}$.
 It is a consequence of (\ref{Pklsigns}) that $c(k,\ell,m)=c(-k,\ell,m)$, and that $k\in{\mathcal S}_\ell\Longleftrightarrow -k\in{\mathcal S}_\ell$.  The argument of $P^{k,0,\ell}F$ in (\ref{pfofbd34}) has Iwasawa $A$-component $a_{(c^2+d^2)\i,t^3\sqrt{c^2+d^2}}$.   By part 2) of Lemma~\ref{lem:estimatesonW} and part 1) of \lemref{lem:Fouriercoeffbound},
\begin{equation}\label{cklWbd}
  c(k,\ell)\,W_\l(a_{k(c^2+d^2)\i,\ell t^3\sqrt{c^2+d^2}}) \ \ = \ \ O_\e(e^{\e|k|-
 \pi|k|/(c^2+d^2)-\pi t^3})
\end{equation}
 for any $\e>0$ (since $|\ell|\sqrt{c^2+d^2}\ge 1$).  Thus the sum over  $k\notin{\mathcal S}_\ell\cup\{0\}$ in (\ref{pfofbd34})  tends to 0 rapidly as $t\rightarrow\infty$, and its estimate remains valid when the sum is restricted to the (finitely many) $k\in {\mathcal S}_\ell\cup\{0\}$.
    In terms of (\ref{Pl000}) and (\ref{PlokandMwl}),
  \begin{multline}\label{pfofbd5}
     \sum_{\srel{k\,\in\,\mathcal S_\ell}{k\,>\,0}}\left[ \sum_{m\,=\,1}^3 c(k,\ell,m)\,\(e^{2\pi i k \theta_\g}+e^{2\pi i (- k) \theta_\g}\)\,
     \phi^{(m)}_\l\(a_{\f{k}{c^2+d^2},\ell t^3\sqrt{c^2+d^2}}\)\right] \\
      + \ \
      \sum_{w\,\in\,\Omega}c(0,\ell;w){\mathcal M}^{\a_2}_{\text{degen},w\l}(a_{(c^2+d^2)\i,\ell t^3\sqrt{c^2+d^2}})
          \ \ \ll \ \ e^{K t^3}\,,  \ t> 1\,;
\end{multline}
 the reason the $m$-sum only includes terms for $m=1,2,3$ is that
 \thmref{thm:To&Nicolas} implies the decay of $\phi^{(m)}_\l\(a_{\f{ k }{c^2+d^2}, \ell t^3\sqrt{c^2+d^2}}\)$, $m=4,5,6$, as $t\rightarrow \infty$ and so they may be omitted.
 \thmref{thm:To&Nicolas} gives the following  large $t>0$ asymptotics for the remaining $\phi_\lambda^{(m)}$:
 \begin{equation}\label{pfofbd6}
 \aligned
    \phi^{(1)}_\l\!\(a_{\f{k}{c^2+d^2},\ell t^3\sqrt{c^2+d^2}}\) \ \ & \sim \ \
    \exp(2\pi t^3|\ell|\sqrt{c^2+d^2} \,+\,2\pi t\,\smallf{3|\ell|^{1/3}|k|^{2/3}}{2\sqrt{c^2+d^2}}
    ) \\   \phi^{(2)}_\l\!\(a_{\f{k}{c^2+d^2},\ell t^3\sqrt{c^2+d^2}}\) \ \ & \sim \ \
    \exp(2\pi t^3|\ell|\sqrt{c^2+d^2}\,+\, 2\pi t\,\smallf{3(-1+i\sqrt{3})|\ell|^{1/3}|k|^{2/3}}{4\sqrt{c^2+d^2}})\\   \phi^{(3)}_\l\!\(a_{\f{k}{c^2+d^2}, \ell  t^3\sqrt{c^2+d^2}}\) \ \ & \sim \ \
    \exp(2\pi t^3|\ell| \sqrt{c^2+d^2}\,+\, 2\pi t\,\smallf{3(-1-i\sqrt{3})|\ell|^{1/3}|k|^{2/3}}{4\sqrt{c^2+d^2}}).
 \endaligned
 \end{equation}
In particular,   for fixed choices of $k,\ell,c$, and $d$ these expressions are linearly independent as functions of $t>0$.

Each summand for fixed $k$ and $m$ in the first line of (\ref{pfofbd5}) has growth as $t\rightarrow\infty$ given by a constant multiple of the appropriate asymptotic in (\ref{pfofbd6})  -- this is because it is impossible for $e^{2\pi i k\theta_\gamma}+e^{2\pi i (-k)\theta_\gamma}=2\cos(2\pi k\theta_\g)$ to vanish for $\ttwo abcd\in SL(2,\Z)$.\footnote{Suppose to the contrary that $\theta_\g=\smallf{ac+bd}{c^2+d^2}\in \pm \f14+\Z$; it is clear $c$ cannot be 0, for then $d=\pm1$. By adding integral multiplies of $(c,d)$ to $(a,b)$ (which does not change the coset representative of $\ttwo abcd \in \G_{\infty}^{(2)}\backslash \G^{(2)}$), we may assume $\smallf{ac+bd}{c^2+d^2}=\f ac-\f{d}{c(c^2+d^2)}=\pm \f14$, or equivalently that $4d=(c^2+d^2)(4a\mp c)$. However, $c^2+d^2>4|d|$ if $|d|>4$, in which case it certainly cannot divide $4|d|$.  Neither can $c^2+9$ divide $12$ (when $|d|=3$), nor $c^2+4$ divide 8 (when $|d|=2$), since $c$ must be odd when $d$ is even.  Is it trivial to see there are no solutions when $|d|\le 1$ and $c\neq 0$.}
 The large $t$ asymptotics of ${\mathcal M}^{\a_2}_{\text{degen}, \l}(a_{(c^2+d^2)\i,\ell t^3\sqrt{c^2+d^2}})$ are determined by (\ref{IKnuasympt}) and (\ref{Mdegendef2}) as
 a constant (depending on $c$, $d$, and $\ell$) times $t^{3(1+\l_1)/2}e^{2\pi t^3|\ell|\sqrt{c^2+d^2}}$.  Because of assumption (\ref{generalposition}), the $\l_i$ are distinct and the last line of (\ref{pfofbd5}) is asymptotic as $t\rightarrow\infty$ to $e^{2\pi t^3|\ell|\sqrt{c^2+d^2}}$ times a linear combination $t^{3/2(1+\lambda_i)}$, $i\in\{1,2,3\}$.

  We now specialize  $(c,d)$ to have   $\sqrt{c^2+d^2} > K/\ell$, so that each individual term in the finite sum (\ref{pfofbd5}) violates the $O(e^{Kt^3})$ bound on the righthand side for large $t$ and hence some cancellation must occur.  However, cancellation is impossible because of the distinct asymptotics of each summand in (\ref{pfofbd5}) (particularly because of the factor of $|k|^{2/3}t$ in the exponentials in (\ref{pfofbd6}), and the absence of such a factor in the contributions from the second line of (\ref{pfofbd5})).  We conclude that the left-hand side of (\ref{pfofbd5}) is identically zero, in particular that   (\ref{Pk000degena2}) holds for all $\ell\neq 0$ and  that
    $c(k,\ell,m)=0$ for all $k,\ell\neq0$ and $m=1,2,3$.
    Applying the same analysis using the   contragredient symmetry (\ref{contragredient}) shows $c(k,\ell,m)=0$ for $m=4,5$ as well, i.e., (\ref{Pk0ellW}) is true for all $k,\ell\neq0$, as was to be shown.
  \bx

\section{Hecke operators}\label{sec:Hecke}

A theorem of Averbuch \cite{Averbuch} asserts that Hecke eigenforms must have moderate growth (this is because the values of Hecke eigenforms at various points satisfy   recursion relations that lead to a polynomial growth estimate).   There is nevertheless a Hecke action on automorphic functions of exponential growth, which   involves increasing the growth rate.  Let us first describe this for $SL(2,\Z)$ and the Hecke operators
\begin{equation}\label{sl2hecke}
    [H_n f](z) \ \ = \ \ \sum_{\srel{ad=n}{b\imod d}}f(\smallf{az+b}{d})
\end{equation}
($H_n$ coincides with the usual Hecke operator $T_n$ up to scaling).
Although $H_n$ of course preserves $SL(2,\Z)$-invariance, it also preserves the weaker condition of periodicity in $\Re{z}$.  Thus
\begin{equation}\label{Hnonqm}
    H_n:e^{-2\pi i m z} \mapsto \sum_{d|n}e^{-2\pi i m nz/d^2}\sum_{b\,=\,1}^de^{-2\pi i m b/d} = \sum_{d|m,n}d\, e^{-2\pi i m n z/d^2}\,.
\end{equation}
From a consideration of the case of $n$ coprime to $m$ (where $d$ must equal 1), one sees the special case of Averbuch's theorem that Hecke operators do not preserve the space of periodic functions satisfying a fixed exponential growth bound.

This leads to some natural questions, in which $\mathcal O=\sum_{n=1}^\infty c_n H_n$ is taken to be a formal infinite linear combination of Hecke operators:
\begin{enumerate}
  \item Can one find coefficients $\{c_n|n\ge 1\}$ such that ${\mathcal O}j=e^{j}$, where $j$ is the classical $j$-function and the sum implicit on the left-hand side converges absolutely?
  \item Given a holomorphic weight zero modular function $f$ for $SL(2,\Z)$ such as $e^j$, can one arrange that ${\mathcal O}f$  equals some power of $j$ (in particular, could it equal $j$ itself)?   Could the sum defining $\mathcal O f$ converge absolutely?
  \item In general, given an eigenfunction $f$ of the full ring of invariant differential operators on $SL(n,\Z)\backslash SL(n,\R)/SO(n,\R)$, can one apply some convergent,  infinite linear combination of Hecke operators $\mathcal O$ to obtain a function ${\mathcal O}f$ that is bounded by some fixed exponential?
\end{enumerate}

The importance  of this last question is that  it would allow us to relax (\ref{thmcond}) in Theorem~\ref{thm:main};   it would also link the notion of weak Maass form (which assumes an exponential bound) to arbitrary eigenfunctions (without presumed bounds).  Let us thus
  formally calculate the action of $\cal O$ on an infinite polar part ${\frak P}=\sum_{m>0}e_m q^{-m}$, where $q=e^{2\pi i z}$.  In order for the sum defining ${\frak P}$ to converge, we must have that $e_m$ decays to zero faster than any decaying exponential in $m$.  Then
\begin{equation}\label{formalhecke1sl2}
\aligned
    {\cal O} {\frak {P}} \ \ & = \ \ \sum_{m,n>0} c_n e_m H_n q^{-m} \ \ = \ \ \sum_{m,n>0} c_n e_m \sum_{d|m,n} d \, q^{-mn/d^2} \\
    & = \ \
    \sum_{k>0}f_k q^{-k}\,, \ \ \ \ \text{where} \ \ \  f_k \ =    \sum_{\srel{m,n>0}{\srel{d|m,n}{mn=kd^2}}} c_n e_m d\,.
\endaligned
\end{equation}
Note that the last expression is symmetric under the interchange $c_n\leftrightarrow e_n$.
For example, consider the simplest case where $\frak{P}=\frak{P}_1=q^{-1}$, which has $e_m=\d_{m=1}$) and hence $f_k=c_k$.  In particular, only the trivial Hecke operator $T_1$ stabilizes  $\frak{P}_1$, and $\mathcal O$ is completely determined by its action on $\frak{P}_1$.

\appendix

\end{document}